\documentclass[a4paper,10pt]{article}
\usepackage{amsmath,amsthm,amssymb,dsfont,graphicx,xspace,epsfig,xcolor}
\usepackage[plain]{fullpage}
\usepackage{thmtools, thm-restate}
\usepackage{hyperref}

\usepackage{algorithm}\usepackage{algorithmic}
\usepackage{tikz}
\usepackage{color}
\usepackage{comment}
\usepackage{authblk,enumitem}

\newtheorem{theorem}{Theorem}[section]
\newtheorem{corollary}[theorem]{Corollary}
\newtheorem{proposition}[theorem]{Proposition}
\newtheorem{lemma}[theorem]{Lemma}
\newtheorem{claim}{Claim}[theorem]

\theoremstyle{definition}

\newtheorem{problem}[theorem]{Problem}
\newtheorem{conjecture}[theorem]{Conjecture}

\newtheorem{case}{Case}

\newcommand{\Ra}{\Rightarrow}

\newcommand{\e}{\mathrm{e}}

\DeclareMathOperator{\rk}{rk}
\DeclareMathOperator{\Inv}{Inv}

\DeclareMathOperator{\sinv}{sinv}
\DeclareMathOperator{\inv}{inv}

\DeclareMathOperator{\cc}{cc}

\DeclareMathOperator{\Bin}{Bin}
\DeclareMathOperator{\bigO}{\mathcal{O}}
\DeclareMathOperator{\UG}{UG}
\DeclareMathOperator{\defect}{def}

\newenvironment{proofclaim}[1][]%
	{\par\noindent {\it Proof of claim}. }{ \hfill$\lozenge$\par\vspace{11pt}}

\title{On the minimum number of inversions to make a digraph $k$-(arc-)strong.}

\author[1]{Julien Duron}
\author[2]{Fr\'ed\'eric Havet}
\author[3]{Florian H\"orsch}
\author[2,4]{Cl\'ement Rambaud}

\affil[1]{\'Ecole normale sup\'erieure de Lyon, LIP, France}
\affil[2]{Universit\'e C\^ote d'Azur, CNRS, Inria, I3S, Sophia Antipolis, France}
\affil[3]{CISPA Saarbrücken, Germany\\florian.hoersch@cispa.de}
\affil[4]{DIENS, \'Ecole normale sup\'erieure, CNRS, PSL University, Paris, France}

\date{}

\begin{document}
\maketitle
\begin{abstract}
The {\it inversion} of a set $X$ of vertices in a digraph $D$ consists of reversing the direction of all arcs of $D\langle X\rangle$. We study $\sinv'_k(D)$ (resp. $\sinv_k(D)$) which is (for some positive integer $k$) the minimum number of inversions needed to transform $D$ into a $k$-arc-strong (resp. $k$-strong) digraph or $+\infty$ if no such transformation exists. 
Note that $\sinv'_k(D) \leq \sinv_k(D)$.
We set $\sinv'_k(n) = \max\{\sinv'_k(D) \mid D~\mbox{is a $2k$-edge-connected digraph of order $n$}\}$. We show the following results where $k$ is a fixed integer for $(i)-(vi)$: 
\begin{itemize}
\item[(i)] $\frac{1}{2} \log (n - k+1) \leq \sinv'_k(n) \leq \log n + 4k -3$ for every $n\geq k$;
\item[(ii)] for any fixed positive integer $t$, deciding whether a given oriented graph $D$ with $\sinv'_k(D)<+\infty$ satisfies $\sinv'_k(D) \leq t$ is NP-complete; 
\item[(iii)] for any fixed positive integer $t$, deciding whether a given oriented graph $D$ with $\sinv_k(D)<+\infty$ satisfies $\sinv_k(D) \leq t$ is NP-complete; 
\item[(iv)] if $T$ is a tournament of order at least $2k+1$, then
$\sinv_k(T) \leq 2k$, 
and $\sinv'_k(T) \leq \frac{4}{3}k+o(k)$;

\item[(v)] $\frac{1}{2}\log(2k+1) \leq \sinv'_k(T)$ for some tournament $T$ of order $2k+1$; 

\item[(vi)] if $T$ is a tournament of order at least $19k-2$ (resp. $11k-2$), then
$\sinv_k(T) \leq 1$ (resp. $\sinv_k(T) \leq 3$);
\item[(vii)] for every $\epsilon>0$, there exists $C$ such that for every positive integer $k$ and every tournament $T$ on at least $2k+1 + \epsilon k$ vertices, we have $\sinv_k(T) \leq C$.
\end{itemize}

%

\medskip

\noindent{}{\bf Keywords:}  inversion; tournament; $k$-strong; $k$-arc-strong.
\end{abstract}

\section{Introduction}
Notation not given below is consistent with \cite{bang2009}. For some positive integer $k$, we denote by $[k]$ the set $\{1,2, \dots, k\}$, and we use $\log$ to refer to the logarithm to the base 2. A {\bf digraph} may contain {\bf digons}, those are pairs of arcs in opposite direction between the same two vertices, but no parallel arcs or loops while an {\bf oriented graph} is a digraph without digons. A {\bf multigraph} may contain parallel edges, but no loops, while a multigraph with no parallel edges is called a {\bf graph}. When the graph or digraph is clear from the context, we use $n$ for its number of vertices.

A {\bf feedback arc set}  in a digraph is a set of arcs whose reversal results in an acyclic digraph. 
Considering an arbitrary ordering $(v_1, \dots , v_n)$ of the vertices of a digraph $D$ and a smallest set among the one containing the forward arcs (i.e. arcs $v_iv_j$ with $i<j$) and the one containing the backward arcs (i.e. arcs $v_iv_j$ with $i>j$), one gets that $D$ has a feedback arc set of size at most $|A(D)|/2$. 
In particular, every oriented graph of order $n$ has a feedback arc set of size at most $\frac{1}{2}\binom{n}{2}$. 
This bound is almost tight : de la Vega~\cite{dlV83} showed that there are infinitely many oriented graphs with no feedback arc set of size less than  $\frac{1}{2}\binom{n}{2} - 1.73 n^{3/2}$.
Finding a minimum cardinality feedback arc set of a given digraph is an important and heavily studied algoritmic problem. It is one of the first problems shown to be NP-hard listed by Karp in~\cite{karp1972}. 
Furthermore, it is hard to approximate. For arbitrary digraphs, the best
known ratio is $\bigO(\log n \log \log n) $ due to Even et al.~\cite{EvenNSS95}. It was proven to be APX-hard by Kann~\cite{KannThesis}. This was later strengthened by Dinur and Safra who showed that, unless P is
equal to NP, it cannot be approximated within a factor of better than about $1.36$~\cite{DiSa05}.
For {\bf tournaments}, which are orientations of complete graphs, the problem was proven to remain NP-complete, independently by Alon~\cite{alonSJDM20} and Charbit, Thomassé, and Yeo~\cite{charbitCPC16}. On the other hand, a polynomial-time approximation scheme was provided by Kenyon-Mathieu and Schudy ~\cite{HowSTOC07}, improving on an earlier  $3$-approximation algorithm by Ailon, Charikar, and Newman~\cite{ACN08}.

\medskip

To make a digraph $D$ acyclic, one can use a different operation from arc reversal, called inversion.
The {\bf inversion} of a set $X$ of vertices consists in reversing the direction of all arcs of $D\langle X\rangle$.
We say that we {\bf invert} $X$ in $D$. The resulting digraph is denoted by $\Inv(D;X)$.  
If $(X_i)_{i\in I}$  is a family of subsets of $V(D)$, then $\Inv(D; (X_i)_{i\in I})$ is the digraph obtained after inverting the
$X_i$ one after another. Observe that this is independent of the order in which we invert the $X_i$~: $\Inv(D; (X_i)_{i\in I})$ is obtained from $D$ by reversing exactly those arcs for which an odd number of the $X_i$ contain its two endvertices.
A {\bf decycling family} of a digraph $D$ is a family $(X_i)_{i\in I}$ of subsets of $V(D)$ such that
 $\Inv(D; (X_i)_{i\in I})$ is acyclic.
A digraph admits a decycling family if and only if it does not contain any digon. 
Indeed, observe that an inversion operation changes the orientation of either none or both of the arcs in a digon. Hence, a digraph containing a digon cannot be made acyclic by inversions.
Conversely, in an oriented graph, the pairs of endvertices of the arcs of a minimal feedback arc-set form a decycling family.
The {\bf inversion number} of an oriented graph $D$, denoted by $\inv(D)$, is the minimum number of inversions needed to transform $D$ into an acyclic oriented graph, that is, the minimum cardinality of a decycling family.
This parameter was first introduced by Belkhechine et al. in \cite{BBBP10} and then studied in several papers~\cite{BCH,PST,inversionejc,APSSW}.
In particular, Belkhechine et al. proved in \cite{BBBP10}  that, for any fixed integer $k$, deciding whether a  given tournament has inversion number at most $k$ is polynomial-time solvable. In contrast, Bang-Jensen et al.~\cite{BCH} proved that deciding whether a given digraph has inversion number~$1$ is NP-complete. This was generalized by Alon et al.~\cite{APSSW}:  for any fixed positive integer $k$, deciding whether a  given digraph has inversion number at most $k$ is NP-complete.
Further, the extremal question of determining the maximum integer $\inv(n)$ of the inversion numbers of all oriented graphs on $n$ vertices has been investigated.
Independently, Aubian et al.~\cite{inversionejc} and Alon et al.~\cite{APSSW} proved
$n - 2\sqrt{n\log n} \leq \inv(n) \leq n - \lceil \log (n+1) \rceil$.

\medskip
The main purpose of this article is to study the possibilities of applying the inversion operation to obtain a different objective than the resulting digraph being acyclic. Instead of making a digraph acyclic, we are interested in making it satisfy a prescribed connectivity property.
A digraph $D$ is {\bf strongly connected} or simply {\bf strong} (resp. {\bf $k$-arc-strong} for some positive integer $k$), if for any partition $(V_1, V_2)$ of $V(D)$ with $V_1, V_2\neq \emptyset$ there is an arc (resp. at least $k$ arcs) with tail in $V_1$ and head in $V_2$.
Similarly, a multigraph $G$ is {\bf connected} (resp. {\bf $k$-edge-connected} for some positive integer $k$), if for any partition $(V_1, V_2)$ of $V(G)$ with $V_1, V_2\neq \emptyset$ there is an edge (resp. at least $k$ edges) with one endvertex in $V_1$ and one endvertex in $V_2$. We further say that $G$ is {\bf $k$-connected} if $|V(G)|\geq k+1$ and $G-S$ is connected for every $S \subseteq V(G)$ with $|S|\leq k-1$.
For a given digraph $D$, we denote by $\UG(D)$ the undirected multigraph that we
obtain by suppressing the orientations of the arcs.  A digraph is {\bf $k$-connected} (resp. {\bf $k$-edge-connected}) if its underlying multigraph $\UG(D)$ is.
Clearly, a digraph $D$ can be made $k$-arc-strong by reversing some arcs if and only if the edges of $\UG(D)$ can be oriented such that the resulting digraph
is $k$-arc-strong.
Robbins' Theorem~\cite{Rob1939} asserts that a graph admits a strong orientation if and only if it is 2-edge-connected, and more generally, Nash–Williams’ weak orientation theorem~\cite{NashW60} asserts that a graph admits a $k$-arc-strong orientation 
if and only if it is $2k$-edge-connected.

We can hence decide in polynomial time whether a given digraph can be transformed into a $k$-arc-strong digraph via arc reversals, applying standard flow algorithms to its underlying graph. Furthermore, it is well-known that, if this is the case, then, by reducing to a minimum-cost submodular flow problem, one can determine, in polynomial time, a minimum set of arcs in $D$ whose reversal gives a $k$-arc-strong digraph, see Section 8.8.4 of~\cite{bang2009} for details.
It is easy to see that the number of necessary arc reversals to make a $2k$-edge-connected digraph $D$ $k$-arc-strong cannot be bounded by a function depending only on $k$. For example, one can consider digraphs that contain a number of sinks which is linear in the number of vertices of the graph, where a {\bf sink} is a vertex with no outgoing arc.
However, Bang-Jensen and Yeo~\cite{BaYe04} proved that for tournaments, the size of such a set is always bounded by a quadratic function of $k$. Precisely, they showed that every tournament on at least $2k+1$ vertices can be made $k$-arc-strong by reversing at most $\frac{1}{2}k(k+1)$ arcs. This result is tight for the transitive tournaments.

We are interested in the problem of using inversions to make a digraph $k$-arc-strong.
A {\bf $k$-arc-strengthening family} of a digraph $D$ is a family $(X_i)_{i\in I}$ of subsets of $V(D)$ such that
 $\Inv(D; (X_i)_{i\in I})$ is $k$-arc-strong.
The {\bf $k$-arc-strong inversion number} of a digraph $D$, denoted by $\sinv'_k(D)$, is the minimum number of inversions needed to transform $D$ into a $k$-arc-strong digraph, that is, the minimum cardinality of a $k$-arc-strengthening family.
We first deal with the extremal behaviour of $\sinv'_k(D)$ for some fixed $k$, that is, we deal with the question of finding the maximum number of necessary inversions to make a $2k$-edge-connected digraph on a fixed number of vertices $k$-arc-strong. To this end, we define $\sinv'_k(n) = \max\{\sinv'_k(D) \mid D~$\mbox{is a $2k$-edge-connected digraph of order $n$}$\}$. It turns out that $\sinv_k'(n)$ is an unbounded, but slowly growing, function. We are able to determine $\sinv_k'(n)$ up to a constant multiplicative factor of roughly 2. More precisely, we show the following result:


\begin{restatable}{theorem}{extrem}\label{thm:extrem}
For any pair of positive integers $k,n$ with $n \geq k+2$, we have 
\begin{equation*}
\frac{1}{2} \log (n - k+1) \leq \sinv'_k(n) \leq \log n + 4k -3.
\end{equation*}
\end{restatable}

Observe that the condition $n \geq k+2$ is necessary for $\sinv_k'(n)$ to be well-defined.
\medskip

 Next, we consider the algorithmic complexity of computing $\sinv_k'(D)$ for a given digraph $D$ and a fixed integer $k$. We show that this problem is NP-hard. More precisely, we show the following, slightly stronger, result. 


\begin{restatable}{theorem}{archard}\label{archard1}
Deciding whether a given oriented graph $D$ satisfies $\sinv'_k(D) \leq t$ is NP-complete for all pairs of positive integers $k$ and $t$.
\end{restatable}

Furthermore, we show that there is little hope to approximate $\sinv'_k(D)$ in polynomial time.

\begin{restatable}{theorem}{approxarc}\label{thm:inapprox-arc}
  Unless P=NP, there do not exist any integers $c_1,c_2,k$ with $1\leq c_1\leq c_2$ and $k \geq 1$ for which there exists a polynomial-time algorithm that can assert whether a given $2k$-edge-connected oriented graph can be made $k$-arc-strong by at most $c_1$ inversions or cannot be made $k$-arc-strong by at most $c_2$ inversions.
\end{restatable}

\medskip

As a related problem, one may also want to make a digraph $k$-strong.
 A digraph $D$ is {\bf $k$-strong} if $|V(D)|\geq k+1$ and for any set $S \subseteq V(D)$ with fewer than $k$ vertices, $D-S$ is strong.
A digraph which can be made $k$-strong by reversing arcs
 is {\bf $k$-strengthenable}.
The $1$-strengthenable digraphs are the $2$-edge-connected ones, because
being $1$-strong is the same as being strong or $1$-arc-strong. 
Thomassen~\cite{Thomassen2015} proved that the $2$-strengthenable digraphs are the $4$-edge-connected digraphs $D$ such that $D-v$ is $2$-edge-connected for every vertex $v \in V(D)$. On the other hand, it is NP-hard to compute the minimum number of arc reversals needed to make a given digraph 2-strong, as shown by Bang-Jensen et al. in ~\cite{BJHK}.
Furthermore, in contrast to the anologous problem for $k$-arc-strengthenable digraphs, for $k\geq 3$, it is NP-complete to decide whether a digraph is $k$-strengthenable. Indeed, Durand de Gevigney~\cite{Dur20} proved that it is NP-complete to decide whether an undirected graph has a $k$-strong orientation for any $k \geq 3$.

It is also natural to use inversions to make a digraph $k$-strong. 
A {\bf $k$-strengthening family} of a digraph $D$ is a family $(X_i)_{i\in I}$ of subsets of $V(D)$ such that
 $\Inv(D; (X_i)_{i\in I})$ is $k$-strong.
The {\bf $k$-strong inversion number} of a $k$-strengthenable digraph $D$, denoted by $\sinv_k(D)$, is the minimum number of inversions needed to transform $D$ into a $k$-strong digraph, that is, the minimum cardinality of a $k$-strengthening family. 
It follows directly from the fact that every $k$-strong digraph is $k$-arc-strong that $\sinv_k(D)\geq \sinv'_k(D)$ holds for every $k$-strengthenable digraph $D$.
In the light of the complexity result of Durand de Gevigney~\cite{Dur20}, it seems difficult to obtain an extremal result in the shape of Theorem~\ref{thm:extrem} for $k$-strong digraphs. However, we give the following complexity results which are the natural analogues of Theorems~\ref{archard1} and~\ref{thm:inapprox-arc}.

\begin{restatable}{theorem}{verhard}\label{verhard}
Deciding whether a given $k$-strengthenable oriented graph $D$ satisfies $\sinv_k(D) \leq t$ is NP-complete for all pairs of positive integers $k$ and $t$.
\end{restatable}
\begin{restatable}{theorem}{approxver}\label{thm:inapprox-vertex}
 Unless P=NP, there do not exist any integers $c_1,c_2,k$ with $1\leq c_1\leq c_2$ and $k \geq 1$ for which there exists a polynomial-time algorithm that can assert whether a given $k$-strengthenable oriented graph can be made $k$-strong by at most $c_1$ inversions or cannot be made $k$-strong by at most $c_2$ inversions.
\end{restatable}

\medskip

In the remainder of this article, we focus on a particular kind of oriented graphs, namely tournaments. 
No tournament on fewer than $2k+1$ vertices is $k$-arc-strong, and thus no tournament on fewer than $2k+1$ vertices is $k$-arc-strengthenable. On the other hand, it is not hard to show that every tournament of order at least $2k+1$ is $k$-strengthenable. Moreover, every tournament on at least $2k+1$ vertices can be made $k$-strong by reversing the orientation of at most $\frac{1}{4}(4k-2)(4k-3)$ arcs, see e.g. \cite{bang2009}, p. 379.

Further, the following improvement was conjectured by Bang-Jensen. See \cite{bang2009}.
\begin{conjecture}[Bang-Jensen, 1994]
Every tournament on at least $2k+1$ vertices can be made $k$-strong by reversing at most $\frac{1}{2}k(k+1)$ arcs.
\end{conjecture}
This would be tight as shown by the transitive tournaments. 
Bang-Jensen, Johansen, and Yeo~\cite{BKY2020} proved a weaker version of this conjecture where $2k+1$ is replaced by $3k-1$.

It is then natural to ask whether or not we can make a tournament $k$-strong or $k$-arc-strong using significantly fewer than $\frac{1}{2}k(k+1)$ inversions. This leads us to consider $M_k= \max\{\sinv_k(T) \mid T~\mbox{tournament of order at least $2k+1$}\}$ and $M'_k= \max\{\sinv'_k(T) \mid T~\mbox{tournament of order at least $2k+1$}\}$.
We show that these numbers are indeed significantly smaller than $\frac{1}{2}k(k+1)$, but cannot be bounded by a constant independent of $k$. More precisely, we show the following result.
\begin{theorem}\label{thm:exttournoi}
    For every sufficiently large integer $k$, we have
$\frac{1}{2} \log(2k+1) \leq M'_k \leq M_k \leq 2k.$
\end{theorem}
The lower bound is obtained by a counting argument.

We also prove that $M_1=M'_1=1$ and $M_2=M'_2=2$ showing that the upper bound is not tight for $k=1,2$. 
We also prove a better upper for $M'_k$ for large values of $k$.

\begin{restatable}{theorem}{borneMprime}\label{thm:upperM'}
$M'_k \leq \frac{4}{3}k+o(k)$. 
\end{restatable}

However, we believe that this bound, as well as the others, is not tight.
\begin{problem}\label{rsedth}
Find better bounds on $M_k$ and $M'_k$.
\end{problem}

We further study the question of how the parameters $\sinv_k$ and $\sinv'_k$  behave when fixing $k$ and considering tournaments whose size is significantly larger than $2k+1$.  As a first result, we  prove that every sufficiently large tournament can be made $k$-strong (and thus also $k$-arc-strong) in one inversion.

\begin{restatable}{theorem}{sone}\label{nk1}
Let $n$ and $k$ be positive integers with $n \geq 19k-2$. If $T$ is a tournament on $n$ vertices, then $\sinv'_k(T) \leq \sinv_k(T) \leq 1$.
\end{restatable}
In fact, it is
relatively straightforward to see that $\sinv_k(T) \leq 1$ for tournaments $T$ of even larger
order relative to $k$ (see Theorem~\ref{thm:s1}).

This leads us to the study of the functions
$m_k(n)= \max \{\sinv_k(T) \mid T~\mbox{tournament of order}~n\}$ and $m'_k(n)= \max \{\sinv'_k(T) \mid T~\mbox{tournament of order}~n\}$ for all $n\geq 2k+1$. We believe that $m_k$ and $m_k'$ have the following monotonic behaviour. 

\begin{restatable}{conjecture}{theconj}\label{conjM}
Let $k$ be a positive integer.
\begin{enumerate}
  \item[(i)] $m_k$ and $m'_k$ are non-increasing mappings.
  \item[(ii)] $M_k =  m_k(2k+1)$ and $M'_k =m'_k(2k+1)$.
\end{enumerate}
\end{restatable}

Note  that (i) implies (ii).
This conjecture is motivated by the fact that one can easily prove that $m_k$ and $m'_k$ are non-increasing for $n>4k-2$. See Section~\ref{conclusion}.
Therefore, in order to approach Problem~\ref{rsedth}, 
it is sufficient to consider tournaments whose order is in the range from $2k+1$ to $4k-2$.

\medskip

Theorem~\ref{nk1} implies that, for every pair of positive integers $k$ and $i$, there is a smallest integer $N_k(i)$ such that
$m_k(n) \leq i$ for all $n\geq N_k(i)$.
Similarly, for every pair of positive integers $k$ and $i$, there is a smallest integer $N'_k(i)$ such that
$m'_k(n) \leq i$ for all $n\geq N'_k(i)$.
Since $m'_k(n) \leq m_k(n)$ for all $k$ and $n$, we have $N'_k(i) \leq N_k(i)$ for all $k$ and $i$.
It is natural to ask the following questions.
\begin{problem}\label{probnk}
Let $k,i,n$ be positive integers, and let $T$ be a tournament on $n$ vertices.

What is the minimum integer $N_k(i)$ such that $\sinv_k(T) \leq i$ if $n \geq N_k(i)$? 

What is the minimum integer $N'_k(i)$ such that $\sinv'_k(T) \leq i$ if $n \geq N'_k(i)$? 
\end{problem}

Theorem~\ref{nk1} states $N'_k(1)\leq N_k(1)\leq 19k-2$.
Since the transitive tournament on $3k-1$ vertices satisfies $\sinv'_k(TT_{3k+1})=\sinv_k(TT_{3k+1})=2$ (Theorem~\ref{thm:TTn}), we have $ N_k(1) \geq N'_k(1) \geq 3k$. We further establish the following better lower bounds for $N_k(1)$ and $N_k'(1)$ by giving simple examples.
\begin{restatable}{proposition}{nkunter}\label{nk1unter}
    For any positive integer $k$, we have $N_k(1)\geq 5k-2$ and $N'_k(1)\geq 4k-2$.
\end{restatable}
We further show that a significantly smaller tournament can still be dealt by three inversions.

\begin{restatable}{theorem}{nkthree}\label{nk3}
    For any positive integer $k$, we have $N'_k(3)\leq N_k(3)\leq 11k-2$.
\end{restatable}

Finally, using probabilistic methods, we prove that every tournament that is a constant factor bigger than $2k$ can be made $k$-strong by a constant number of inversions. More precisely, we prove the following result.
\begin{restatable}{theorem}{pluseps}\label{thm:2+eps}
    There exists a function $f\colon \mathbb{R}_{>0} \to \mathbb{N}$ such that for every $\epsilon>0$ and every positive integer $k$, 
    if $T$ is an $n$-vertex tournament with $n \geq 2k+1 +\epsilon k$, then $\sinv_k(T) \leq f(\epsilon)$.
\end{restatable}


\medskip

The fact that $m_k(n)=1$ for $n$ sufficiently large implies that the set ${\cal F}_k$ of tournaments $T$ such that $\sinv_k(T) >1$ is finite.
This implies that computing $\sinv_k$ and $\sinv'_k$ for fixed $k$ can be done in polynomial time for tournaments.

\begin{corollary}\label{cor:sinv_k-poly}
Let $k$ be a positive integer. We can compute $\sinv_k(T)$ (resp. $\sinv'_k(T)$) for a given tournament $T$ on $n$ vertices in $O(n^{7/2})$ time (resp. $O(n^2)$ time).
\end{corollary}
\begin{proof}
We first check in constant time, whether $T$ is in ${\cal F}_k$. If yes, then we can return $\sinv_k(T)$ and $\sinv'_k(T)$ that may be stored in some precomputed table.
If not, then $\sinv'_k(T) \leq \sinv_k(T) \leq 1$. We then check in $O(n^{7/2})$ time using algorithms due to Even and Tarjan~\cite{EvenTarjan75} or Galil~\cite{Galil80} based on flow (resp. $O(n^2)$ time using Mansour and  Schieber's~\cite{MANSOUR198976} algorithm based on flow or Gabow's algorithm~\cite{Gabow91} based on matroids) whether $T$ is $k$-strong (resp. $k$-arc-strong) to determine whether $\sinv_k(T)$ (resp. $\sinv'_k(T)$) is equal to $0$ or $1$. 
\end{proof}


This article is structured as follows. We give some notation and a collection of preliminary results in Section~\ref{prel}. In Section~\ref{sec:oriented}, we prove Theorem~\ref{thm:extrem}. In Section~\ref{sec:complexity}, we prove the complexity results, namely Theorems~\ref{archard1} to~\ref{thm:inapprox-vertex}. The results on tournaments are contained in Section~\ref{sec:M}, in which we prove Theorem~\ref{thm:exttournoi}, and in Section~\ref{sec:upper_bound_Mkn}, in which we prove Theorems~\ref{thm:s1},~\ref{nk1},~\ref{nk3} and~\ref{thm:2+eps} and Proposition~\ref{nk1unter}.

\section{Preliminaries}\label{prel}

In Section~\ref{not}, we introduce some formal notation and in Section~\ref{prelres}, we give a collection of auxiliary results.

\subsection{Notation}\label{not}

 A {\bf mixed graph} $G = (V, E, A)$ is a triple consisting of a set $V$ of elements called {\bf vertices}, a set $E$ of unordered pairs of vertices  called {\bf edges}, and a set $A$ of ordered pairs of vertices called {\bf arcs}.
Hence a multigraph can be seen as a mixed graph whose arc set is empty, while a digraph can be seen as a mixed graph whose edge set is empty and that does not contain parallel arcs.
Given a mixed graph $G$, its vertex set is denoted by $V(G)$, its edge set is denoted by $E(G)$ and its arc set by $A(G)$.
The {\bf underlying graph} $\UG(G)$ of a mixed graph $G$ is the undirected multigraph that we obtain by suppressing the orientations of the arcs (i.e. replacing each arc by an edge between its two endvertices). 
An {\bf orientation} of a mixed graph $G$ is a digraph obtained by giving an orientation to every edge, i.e.
replacing each edge by one of the two arcs between its endvertices.

Let $S,S'$ be disjoint sets of vertices in a mixed graph $G$. 
We denote by $\delta_G(S,S')$ the set of edges of $E(G)$ with exactly one endvertex in $S$ and one endvertex in $S'$. We use $\delta_G(S)$ for $\delta_G(S,V(G)\setminus S)$.
We also denote by $\delta^+_G(S,S')$ (resp. $\delta^-_G(S,S')$) the set of arcs in $A(G)$ with tail (resp. head) in $S$ and head (resp. tail) in $S'$. Again, we use $\delta^+_G(S)$ (resp. $\delta^-_G(S)$) for $\delta^+_G(S,V(G)\setminus S)$ (resp. $\delta^-_G(S,V(G)\setminus S)$).
The {\bf degree} (resp. {\bf out-degree}, {\bf in-degree}) of $S$ in $G$ is $d_G(S) = |\delta_G(S)|$ 
(resp.
$d^+_G(S) = |\delta^+_G(S)|$, $d^-_G(S) = |\delta^-_G(S)|$). We further use $d_G(S,S') = |\delta_G(S,S')|$, 
$d^+_G(S,S') = |\delta^+_G(S,S')|$, and 
$d^-_G(S,S') = |\delta^-_G(S,S')|$.
Let {$N_G(S)$ (resp. $N_G^+(S)$, $N_G^-(S)$)} denote the set of vertices in $V(G)\setminus S$ that are incident to at least one edge (arc) in $\delta_G(S)$ (resp. $\delta_G^+(S)$, $\delta_G^-(S)$).
The {\bf degree} (resp. {\bf out-degree}, {\bf in-degree}) of a vertex $v$ in $G$ is $d_G(v) = d_G(\{v\})$ (resp. 
$d^+_G(v) = d^+_G(\{v\})$, $d^-_G(v) = d^-_G(\{v\})$). 
For simplicity, when the mixed graph is clear from the context, the subscript $G$ is omitted.

\medskip

Let $D$ be a digraph.
A {\bf sink} (resp. {\bf source}) in a digraph is a vertex with out-degree $0$ (resp. in-degree $0$).
We say that $D$ is {\bf eulerian} if $d_D^-(v) = d_D^+(v)$ holds for every vertex $v$.
Let $A$ and $B$ be two disjoint sets in $D$.
If $ab$ is an arc for all pairs $(a,b)$ in $A\times B$, then we write $A\Ra B$.

Let $u,v$ be two distinct vertices in $D$. The {\bf strong-connectivity} from $u$ to $v$ in $D$, denoted by $\kappa_D(u,v)$, is the maximal number $\alpha$ such that $D-X$ contains a $(u,v)$-path  for every $X \subseteq V(D)\setminus \{u,v\}$ with $|X|\leq \alpha-1$. For some $S \subseteq V(D)$ and positive integer $k$, we say that $S$ is {\bf $k$-strong in $D$} if $\kappa_D(u,v)\geq k$ for all $u,v \in S$.

\subsection{Preliminary results}\label{prelres}
We use several times without explicitly mentioning it that for every positive integer $k$, there is a $k$-strong tournament on $2k+1$ vertices. One example for such a tournament is the socalled rotative tournament whose vertex set is $[2k+1]$ and where an arc is oriented from $i$ to $j$ if $i-j\in [k] ({\rm mod }~2k+1)$.

\medskip

We need the following simple result that allows us to extend a set which is $k$-strong in a digraph.
\begin{proposition}\label{lem:kstrong+}
Let $D$ be a digraph, let $S$ be a $k$-strong set in $D$ and let $v \in V(D)\setminus S$. If $v$ has at least $k$ in-neighbours in $S$ and at least $k$ out-neighbours in $S$, then $S \cup \{v\}$ is $k$-strong in $D$.
\end{proposition}
\begin{proof}
    Suppose for the sake of a contradiction that there is a set $X\subseteq V(D)$ of size at most $k-1$ and a pair $(s,t)$ of vertices in $(S \cup \{v\}) \setminus X$ such that there is no $(s,t)$-path in $D-X$.
    If $s,t \in S$, then this contradicts the fact that $S$ is $k$-strong in $D$.
    Thus exactly one of $s$ and $t$ is $v$. Without loss of generality, suppose $s \in S$ and $t=v$.
    Since $|N^-(v) \cap S| \geq k$, $v$ has an in-neighbour $t_0 \in S-X$, and as $S$ is $k$-strong in $D$, there is an $(s,t_0)$-path in $D-X$, and so there is an $(s,t)$-path in $D-X$, a contradiction.
\end{proof}
The well-known following result is helpful for applying Proposition~\ref{lem:kstrong+}. It follows directly from Proposition~2.2.2 of \cite{Tour-book}.
\begin{proposition}\label{4k-2}
    Let $T$ be a tournament on at least $4k-1$ vertices for some positive integer $k$. Then there exists some $v \in V(T)$ with $\min\{d_T^+(v),d_T^-(v)\}\geq k$.
\end{proposition}


We need the following simple characterization of $k$-arc-strong tournaments of order $2k+1$.
\begin{proposition}[Folklore]\label{eul}
A tournament on $2k+1$ vertices is $k$-arc-strong if and only if it is eulerian.
\end{proposition}

For sake of completeness, we give a short proof of this fact.

\begin{proof}
    Let $T$ be a tournament on $2k+1$ vertices.
    If $T$ is $k$-arc-strong, then every vertex $u$ in $T$ has in- and out-degree at least $k$, but since $d_T^+(u)+d_T^-(u) =2k$, we deduce that $d_T^+(u)=d_T^-(u)$ and so $T$ is eulerian.
    Reciprocally, if $T$ is eulerian, then $d_T^+(u)=d_T^-(u)=k$ for every vertex $u$ of $T$.
    Now consider a partition $(V_1,V_2)$ of $V(T)$ where $V_1$ and $V_2$ are non-empty.
    Then the number of arcs from $V_1$ to $V_2$ is at least 
    \begin{align*}
        \sum_{v \in V_1} d_T^+(v) - \binom{|V_1|}{2} 
        &\geq |V_1| \left( k - \frac{|V_1|-1}{2} \right) \\
        &\geq k
    \end{align*}
    and so $T$ is $k$-arc-strong.
\end{proof}

We need the following orientation property of mixed graphs that can be found in \cite{FF}.
\begin{proposition}\label{prop:eul-or}
    Let $H$ be a mixed graph whose underlying graph is eulerian. Then $H$ has an eulerian orientation if and only if $d_{H}(S)\geq d_{H}^+(S)-d_{H}^-(S)$ for all $S \subseteq V(H)$.
\end{proposition}

Finally, we state two basic tools from probability theory.
\begin{proposition}[Union Bound]\label{union}
    Let $E_1,\ldots,E_\ell$ be a set of events in a random experiment and $E$ the event that at least one of $E_1,\ldots,E_\ell$ occurs. Then $\Pr(E)\leq \sum_{i=1}^\ell \Pr(E_i)$.
\end{proposition}

\begin{proposition}[Chernoff's Bound]\label{chernoff}
    If $X$ is a random variable following a binomial law with parameters $p \in [0,1]$ and $n \geq 0$, then
    for every $\epsilon \in [0,1]$
    \[
    \Pr[X \leq (1-\epsilon)pn] \leq \exp \left(-\frac{\epsilon^2}{2} pn\right).
    \]
\end{proposition}

We refer the reader to~\cite[Part II, Section 5]{molloy2002graph} for an introduction to the probabilistic method, including proofs of Propositions~\ref{union} and~\ref{chernoff}.

\section{ Bounds on \texorpdfstring{$\sinv'_k(n)$}{sinv'k(n)}}\label{sec:oriented}
This section is dedicated to the extremal results we have on the $k$-arc-strong inversion number. In particular, we prove Theorem~\ref{thm:extrem}, which we restate.

\extrem*

First, we prove the upper bound.
Let $d$ be a positive integer.
A multigraph $G$ is said to be {\bf $d$-degenerate} if every submultigraph of $G$ has a vertex of degree at most $d$.
Every $d$-degenerate graph admits a {\bf $d$-degenerate ordering}, that is an ordering $(v_1, \dots, v_n)$ of the vertices of $G$ such that every vertex has at most $d$ neighbours with lower indices.

We shall need the following proposition, which, as observed in \cite{BJHK}, is a direct consequence of Corollary~2 in~\cite{zoltans06}.
\begin{proposition}\label{prop:nash_williams_more_precise}
Let $G$ be a multigraph that has a $k$-arc-connected orientation for some positive integer $k$
and let $(e_1, f_1),\dots ,(e_t, f_t)$ be a collection of pairwise disjoint pairs of parallel edges in G. Then G has
a $k$-arc-connected orientation in which $e_i$ and $f_i$ are oriented in opposite directions for $i = 1,\dots , t$.
\end{proposition}

We use the following result which has recently been proved in a related paper by the second, third, and fourth author \cite{havet2024diameter}.
\begin{theorem}[\cite{havet2024diameter}]\label{theorem:bound_diam_L_degenerate}
Let $G$ be an $n$-vertex $d$-degenerate graph.
For any two orientations $\vec{G}_1,\vec{G}_2$ of $G$, one can transform $\vec{G}_1$ into $\vec{G}_2$ by inverting at most $\log n + 2d-1$ sets.
\end{theorem}

Let $G$ be a multigraph and let $S$ be a subset of $V(G)$ with $\emptyset \subsetneq S \subsetneq V(G)$.
The {\bf $S$-cut} in $G$, denoted by $\delta_G(S)$, is the set of all the edges $uv$ in $G$ with $u \in S$ and $v \not\in S$.
A {\bf cut} in $G$ is an $S$-cut for some non-trivial set $S$ (that is $S\neq \emptyset$ and  $S \neq V(G)$).
Note that a multigraph is $k$-edge-connected if and only if there is no cut in $G$ of size at most $k-1$.
For all pair of distinct vertices $u,v$ in $G$, a {\bf $(u,v)$-cut} in $G$ is a cut $\delta_G(S)$ of $G$ with $u \in S$ and $v \not\in S$.
%
%
A multigraph $G$ is {\bf minimally $k$-edge-connected} if it is $k$-edge-connected and 
$G\setminus e$ is not $k$-edge-connected 
for any edge $e \in E(G)$.
Equivalently, for every edge $uv$, there is a cut of size exactly $k$ containing $uv$.

The next result gives an upper bound on the number of edges in a multigraph where each edge is contained in a small cut.
\begin{lemma}\label{lemma:min_k_connected_mad}
    Let $n$ and $k$ be positive integers with $n \geq k+1$, and let $G$ be an $n$-vertex multigraph.
    If for every edge $uv \in E(G)$ there is a $(u,v)$-cut of size at most $k$ in $G$, then $G$ has at most $k(n-1)$ edges.
\end{lemma}
\begin{proof}
    Let $F_1,\ldots,F_k$ be a set of $k$ edge-disjoint spanning forests in $G$ such that $\sum_{i=1}^k|E(F_i)|$ is maximized. 
    If there is an edge $e=uv \in E(G)\setminus \bigcup_{i=1}^k E(F_i)$, then, as $e$ cannot be added to $E(F_i)$, we obtain that $F_i$ contains a $(u,v)$-path for $i\in [k]$. 
    Hence $G\setminus e$ does have any $(u,v)$-cut whose size is smaller than $k$, a contradiction. 
    This yields $|E(G)|=\sum_{i=1}^k|E(F_i)|\leq k(n-1)$.
\end{proof}

We are now ready to prove the upper bound in Theorem~\ref{thm:extrem}.
\begin{lemma}\label{extremsup}Let $n$ and $k$ be positive integers with $n \geq k+1$. Then,
    for every $2k$-edge-connected $n$-vertex digraph $D$, $\sinv'_k(D) \leq \log n + 4k-3$.
\end{lemma}

\begin{proof}
    Without loss of generality, suppose that $D$ is minimally $2k$-edge-connected.
    Then by Lemma~\ref{lemma:min_k_connected_mad}, every subgraph of $\UG(D)$ has average degree smaller than $2k$.
    This implies that $\UG(D)$ is $(2k-1)$-degenerate.
    Let $D_0$ be the subdigraph obtained from $D$ by removing all digons.
    By Proposition~\ref{prop:nash_williams_more_precise}, there is an orientation $D'_0$ of $\UG(D_0)$ such that together with the digons of $D$, this digraph is $k$-arc-strong.
    By Theorem~\ref{theorem:bound_diam_L_degenerate}, there is a set $\mathcal{X}$ of at most $\log n + 2(2k-1) -1$ inversions transforming $D_0$ into $D'_0$. Since digons are preserved by inversions, we deduce that $\Inv(D;\mathcal{X})$ is $k$-arc-strong, and so $\sinv'_k(D) \leq \log n + 4k-3$.
\end{proof}

In order to prove the lower bound in Theorem~\ref{thm:extrem}, we first need the following intermediate result.

\begin{lemma}\label{sizet}
    For every positive integer $t$, there exists a 2-edge-connected digraph $D$ on $2^{t-1}+1+\binom{2^{t-1}+1}{2}$ vertices with a specified vertex $s$ such that  each digraph obtained from $D$ by applying at most $t-1$ inversions contains a sink or source distinct from $s$.
\end{lemma}
\begin{proof}
    Let first $S$ be a set of $2^{t-1}+1$ vertices. Now we obtain the digraph $D$ by adding a vertex $v_{\{s_1,s_2\}}$ as well as the arcs $s_1v_{\{s_1,s_2\}}$ and $s_2v_{\{s_1,s_2\}}$ for every $\{s_1,s_2\}\subseteq S$ . It is easy to see that $D$ is 2-edge-connected and that the number of vertices of $D$ is $2^{t-1}+1+\binom{2^{t-1}+1}{2}$. Now let $D'$ be obtained from $D$ by inverting a collection of $t-1$ sets $X_1,\ldots,X_{t-1} \subseteq V(D)$. As $|S|=2^{t-1}+1$, there exist distinct vertices $s_1,s_2 \in S$  such that for $i=1,\ldots,t-1$, we have either $\{s_1,s_2\}\subseteq X_i$ or $\{s_1,s_2\}\cap X_i = \emptyset$.
    Then in each of the $t-1$ inversions, either both or none of the arcs incident to $v_{\{s_1,s_2\}}$ are inverted. We obtain that $v_{\{s_1,s_2\}}$ is either a source or a sink in $D'$. The statement hence follows for an arbitrary $s \in S$.
\end{proof}

The next lemma gives a construction of graphs of arbitrary size that need a significant amount of inversions to become strong.

\begin{lemma}\label{arbn1}
    For every positive integer $n\geq 3$, there is a 2-edge-connected digraph $D$ on $n$ vertices such that any digraph obtained from $D$ by applying at most $\frac{1}{2}\lceil\log n \rceil-1$ inversions contains a sink or a source.
\end{lemma}

\begin{proof}
For $n=3,4$, the statement is clearly true. We may therefore suppose that $n \geq 5$ and hence $\frac{3}{2 \sqrt{2}\sqrt{n}}\leq \frac{1}{2}$.
    Let $n'=2^{\frac{1}{2}\lceil\log n\rceil-1}+1+\binom{2^{\frac{1}{2}\lceil\log n\rceil-1}+1}{2}$. By Lemma~\ref{sizet}, there is a digraph $D'$ on $n'$ vertices together with a vertex $s \in V(D')$ such that every graph obtained from $D'$ by applying at most $\frac{1}{2}\lceil\log n \rceil-1$ inversions contains a sink or source distinct from $s$.

    Next observe that 
    \begin{align*}
        n'&=2^{\frac{1}{2}\lceil\log n\rceil-1}+1+\binom{2^{\frac{1}{2}\lceil\log n\rceil-1}+1}{2}\\
        &=2^{\frac{1}{2}\lceil\log n\rceil-1}+1+\frac{1}{2}(2^{\frac{1}{2}\lceil\log n\rceil-1}+1)2^{\frac{1}{2}\lceil\log n\rceil-1}\\
        &\leq 2^{\frac{1}{2}\log n-\frac{1}{2}}+1+\frac{1}{2}(2^{\frac{1}{2}\log n-\frac{1}{2}}+1)2^{\frac{1}{2}\log n-\frac{1}{2}}\\
        &= \frac{\sqrt{n}}{\sqrt{2}}+1+\frac{1}{2}\left(\frac{\sqrt{n}}{\sqrt{2}}+1\right)\frac{\sqrt{n}}{\sqrt{2}}\\
        &=\frac{n}{4}+\frac{3}{2 \sqrt{2}}\sqrt{n}+1\\
        &=\frac{n}{4}+\frac{3}{2 \sqrt{2}\sqrt{n}}n+1\\
        &\leq \frac{n}{4}+\frac{n}{2}+\frac{n}{4}\\
        &=n.
    \end{align*}

We now obtain $D$ from $D'$ by adding a new set of $n-n'$ vertices and for each of them, adding a digon linking it to $s$. By construction, $D$ has $n$ vertices. Further, as the same property holds for $D'$, in any graph obtained from $D$ by at most $\frac{1}{2}\lceil\log n \rceil-1$ inversions, one of the vertices in $V(D')\setminus \{s\}$ is a source or a sink.
\end{proof}

We are now ready to prove the lower bound in Theorem~\ref{thm:extrem}.

\begin{lemma}\label{extreminf}
    For every pair of positive integers $n,k$ with $n \geq k+2$, there is a $2k$-edge-connected digraph $D$ on $n$ vertices such that $\sinv'_k(D) \geq \frac{1}{2}\log(n -k+1)$.
\end{lemma}

\begin{proof}
By assumption, we have $n-k+1\geq 3$.
    Hence, by Lemma~\ref{arbn1}, there exists a 2-edge-connected digraph $D_0$ on $n-k+1$ vertices such that any digraph obtained from $D_0$ by inverting fewer than $\frac{1}{2} \log(n-k+1)$ sets has a sink or a source. 
    Now let $D$ be the digraph obtained 
    from $D_0$ by adding a set $S$ of $k-1$ vertices, for any pair of vertices in $S$ a digon linking them, and for every vertex in $S$ and every vertex in $V(D_0)$ a digon linking these vertices. 
    %
    One can check that $D$ has $n$ vertices and is $2k$-edge-connected.
    
    Now consider a family of subsets $\cal X$ such that
    $D'=\Inv(D; {\cal X})$ is $k$-arc-strong. Any vertex $v\in V(D_0)$ is linked to the vertices of $S$ by digons in $D$. Since an inversion transforms a digon into a digon, it is also connected to the vertices of $S$ by digons in $D'$. Hence, in $D'$, $v$ has at least one in-neighbour and one out-neighbour in $V(D_0)$. In other words, the subdigraph of $D'$ induced by $V(D_0)$ has no source and no sink. Thus $|{\cal X}| \geq \frac{1}{2}\log(n-k+1)$.
    
     Therefore $\sinv'_k(D) \geq \frac{1}{2}\log(n-k+1)$.
\end{proof}

     Finally, Lemmas~\ref{extremsup} and~\ref{extreminf} imply Theorem~\ref{thm:extrem}.

\section{Complexity of computing \texorpdfstring{$\sinv_k$}{sinvk} and \texorpdfstring{$\sinv'_k$}{sinv'k}\label{sec:complexity}}

In this section, we deal with the complexity of computing the parameters $\sinv_k(D)$ and $\sinv'_k(D)$ for a given digraph $D$. More concretely, we prove Theorems~\ref{archard1},~\ref{thm:inapprox-arc},~\ref{verhard}, and~\ref{thm:inapprox-vertex}.

In order to prove these results in this section, we shall use another well-studied hypergraph parameter. Let $H$ be a hypergraph. For some positive integer $k$, a {\bf $k$-colouring} of $H$ is a mapping $\phi:V(H)\rightarrow [k]$ such that every hyperedge $e \in E(H)$ contains two vertices $u$ and $v$ with $\phi(u)\neq \phi(v)$. The {\bf chromatic number} of $H$, denoted  by $\chi(H)$ is the smallest integer $k$ such that $H$ admits a $k$-colouring. Generalizing the definition for graphs, for some $S \subseteq V(H)$, the {\bf $S$-cut}, denoted by $\delta_H(S)$, is the set of hyperedges $e\in E(H)$ satisfying $e\cap S\neq \emptyset$ and $e\setminus S\neq \emptyset$.
A {\bf cut cover} of $H$ is a collection $X_1,\ldots,X_t$ of subsets of $V(H)$ such that $\cup_{i=1}^t \delta_H(X_i)=E(H)$.
The {\bf cut covering number} of $H$, denoted by $\cc(H)$, is the minimum integer $t$ such that there is a cut cover of $H$ of size $t$. 
The following well-known result shows a close relationship between the cut covering number and the chromatic number. It was proved in \cite{https://doi.org/10.1002/jgt.3190010208} for graphs and literally generalizes to hypergraphs.

\begin{proposition}[Harary, Hsu, and~Miller~\cite{https://doi.org/10.1002/jgt.3190010208}]\label{prop:cc-chi}
For any hypergraph $H$, we have $\cc(H)=\lceil\log(\chi(H))\rceil$.
\end{proposition}

It is well-known that the problem of deciding whether a given graph can be coloured with $t$ colours is NP-hard for any $t\geq 3$ (see \cite{schaefer1978complexity}). Moreover, the problem of deciding whether a hypergraph is 2-colourable is also NP-hard (see \cite{schaefer1978complexity}). It hence follows from Proposition~\ref{prop:cc-chi} that deciding whether the cut covering number of a given hypergraph is at most $t$ is NP-hard for any integer $t\geq 1$.
We shall show a reduction from this problem to the one of deciding whether $\sinv'_k(D) \leq t$  for a given oriented graph $D$.

We need the following result of Guruswami, Hastad, and Sudan~\cite{892074}.
\begin{proposition}[Guruswami, Hastad, and Sudan~\cite{892074}]\label{zuck}
Unless P=NP, there do not exist any integers $c_1,c_2$ with $2\leq c_1\leq c_2$ for which there exists a polynomial-time algorithm that distinguishes a $c_1$-colourable hypergraph from a not $c_2$-colourable hypergraph.
\end{proposition}

As a consequence, we directly obtain a negative result concerning the approximation of the cut covering number from Proposition \ref{prop:cc-chi}.

\begin{proposition}\label{approcc}
Unless P=NP, there do not exist any integers $c_1,c_2$ with $1\leq c_1\leq c_2$ for which there exists a polynomial-time algorithm that distinguishes a hypergraph of cut covering number at most $c_1$ from a hypergraph of cut covering number more than $c_2$.
\end{proposition}

We shall also use the following notation.
Given a digraph $D$ and a set $X$ of vertices in $D$, we denote by $\partial_D(X)$ the set of edges of $\UG(D)$ with exactly one endvertex in $X$:
$\partial_D(X) = \delta_{\UG(D)}(X)$. 

\subsection{Inversions  to become \texorpdfstring{$k$}{k}-arc-strong}\label{drei}

In this subsection, we prove Theorems~\ref{archard1} and \ref{thm:inapprox-arc}. The two proofs are based on a reduction given by the following lemma.

\begin{lemma}\label{lem:reduc2}
Given a hypergraph $H$ and a positive integer $k$, one can construct in polynomial time an oriented graph $D$ such that $\sinv'_k(D)=\cc(H)$ and $|V(D)|=|V(H)|+(2k+1)|E(H)|+2k+1$.
\end{lemma}

\begin{proof}
Let $<$ be an arbitrary ordering on $V(H)$.
We let $V(D)$ be the disjoint union of $V(H)$, sets of vertices  $Z_e=\{z^1_e,\ldots,z_e^{2k+1}\}$ for every $e \in E(H)$, and a set of vertices $W=\{w_1,\ldots,w_{2k+1}\}$. We add arcs to $D$ such that $D\langle W\rangle$ is a $k$-arc-strong tournament. We let $A(D)$ contain an arc from $v$ to $w_i$ for every $v \in V(H)$ and every $i=1,\ldots,k$ and an arc from $w_i$ to $v$ for every $v \in V(H)$ and every $i=k+1,\ldots,2k$. Next for every $e \in E(H)$, we add arcs to $D$ such that $D\langle Z_e\rangle$ is a $k$-arc-strong tournament. Finally, for every $e \in E(H)$, consider $u,v\in e$ with $u<v$ and $v<v'$ for all $v'\in e\setminus \{u,v\}$. We add the arcs $uz_e^{i}$ and ${z_e^{i}}v$ for $i=2,\ldots,k$ and the arc $v'{z_e^{1}}$ for all $v'\in e$. This finishes the description of $D$. Observe that $|V(D)|=|V(H)|+(2k+1)|E(H)|+2k+1$. 
We now show that $\sinv'_k(D)=\cc(H)$.

\smallskip

Let us first show that $\sinv'_k(D)\leq \cc(H)$. 
Let $(X_1,\ldots,X_{\cc(H)})$ be a minimum-size cut cover of $H$. 
For every $e\in E(H)$, there is a smallest $\alpha(e)$ such that $e \cap X_{\alpha(e)}$ and $e \setminus  X_{\alpha(e)}$ are nonempty. 
Now for $i\in [\cc(H)]$, let $X'_i=X_i\cup\{z_e^1:\alpha(e)=i\}$. 
Let $D'$ be the digraph obtained from $D$ by inverting $X'_1,\ldots,X'_{\cc(H)}$. We will show that $D'$ is $k$-arc-strong. 
Let $S \subseteq V(D)$ be non-empty and suppose that $\min\{d_{D'}^-(S),d_{D'}^{+}(S)\}<k$ .
By symmetry, we may suppose that $S \cap W \neq \emptyset$. 
As $D'\langle W\rangle=D\langle W\rangle$ is $k$-arc-strong, we obtain $W \subseteq S$. 
Next observe that in $D'$ every $v \in V(H)$ is incident to at least $k$ arcs coming from $W$ and $k$ arcs going to $W$. Thus $V(H)\subseteq S$. Now consider some $e \in E(H)$ and let $u,v\in e$ with $u<v$ and $v<v'$for all $v' \in e\setminus \{u,v\}$. 
As $D'\langle Z_e\rangle=D\langle Z_e\rangle$ and $D\langle Z_e\rangle$ is $k$-arc-strong by construction, we obtain that either $Z_e \subseteq S$ or $Z_e \cap S=\emptyset$.
Observe that, by construction, $z_e^1$ has an in-neighbour and an out-neighbour in $V(H)$ in $D'$. 
Further observe that $A(D')$ contains the arcs $uz_e^{i}$ and $z_e^{i}v$ for all $i \in \{2,\ldots,k\}$.
Thus $D'$ contains $k$ arcs from $V(H)$ to $Z_e$ and $k$ arcs from $Z_e$ to $V(H)$, and we obtain that $Z_e\subseteq S$.
This yields $S=V(D')$, so $D'$ is $k$-arc-strong. Thus $\sinv'_k(D)\leq \cc(H)$.

\smallskip

Let us now show that $\cc(H)\leq\sinv'_k(D)$. 
Let $(X_1,\ldots, X_t)$ be a minimum-size $k$-arc-strengthening family of $D$.
Let $D' =\Inv(D ; (X_i)_{i\in [t]})$ and for $i\in[t]$, let $X_i'=X_i \cap V(H)$.
We claim that $(X'_1, \dots, X'_t)$ is a cut cover of $H$.
In other words, for every $e \in E(H)$, there is some $i \in [t]$ such that $X'_i\cap e \neq \emptyset$ and $e \setminus X_i'\neq \emptyset$.

Suppose otherwise. It follows that there is exactly one arc entering $Z_e$ incident to $z_e^{j}$ in $D'$, for every $j=2,\ldots,k$. As $D'$ is $k$-arc-strong, we obtain that in $D'$, there exists at least one arc entering $Z_e$ incident to $z_e^{1}$ and at least one arc entering $Z_e$ incident to $z_e^{1}$. By assumption, for every $i \in[t]$ either $e \subseteq X_i$ or $e \cap X_i =\emptyset$ holds. Thus we deduce that for every $j \in [k]$, either all the arcs corresponding to edges in $\partial_{D}(Z_e)\cap \partial_{D}(z_e^{j})$ are inverted or none of them are, a contradiction.
Therefore, $(X_1',\ldots,X_t')$ is a cut cover of $H$, and so  $\cc(H)\leq t=\sinv'_k(D)$.
\end{proof}

Let $k$ and $t$ be two positive integers.
Lemma~\ref{lem:reduc2} and the NP-hardness of deciding whether a given hypergraph $H$ satisfies $\cc(H)\leq t$ for all $t\geq 1$ directly imply that deciding whether a given oriented graph $D$ satisfies $\sinv'_k(D) \leq t$ is NP-hard.
This problem is clearly in NP because a family $\mathcal{X}$ such that $\Inv(D;\mathcal{X})$ is $k$-arc-strong is a polynomial-size certificate that $\sinv'_k(D) \leq |\mathcal{X}|$.  
Hence it is NP-complete, which proves Theorem~\ref{archard1}.

Moreover, Proposition~\ref{approcc} and Lemma~\ref{lem:reduc2}  imply Theorem~\ref{thm:inapprox-arc}.

\subsection{Inversions to become \texorpdfstring{$k$}{k}-strong}\label{zwei}

In this subsection, we prove Theorems~\ref{verhard} and \ref{thm:inapprox-vertex}. The two proofs are based on a reduction given by the following lemma.

\begin{lemma}\label{lem:reduc3}
Given a hypergraph $H$ and a positive integer $k$, one can construct in polynomial time a $k$-strengthenable oriented graph $D$ such that $\sinv_k(D)=\cc(H)$ and $|V(D)|=|V(H)|+(2k+1)|E(H)|+2k+1$.
\end{lemma}

\begin{proof}
Let $V(D)$ be the disjoint union of $V(H)$, sets of vertices  $Z_e=\{z^1_e,\ldots,z_e^{2k+1}\}$ for every $e \in E(H)$, and a set of vertices $W=\{w_1,\ldots,w_{2k+1}\}$. 
We add arcs to $D$ such that $D\langle W\rangle$ is a $k$-strong tournament.
Let $A(D)$ contain an arc from $v$ to $w_i$ for every $v \in V(H)$ and every $i\in[k]$ and an arc from $w_i$ to $v$ for every $v \in V(H)$ and every $i\in \{k+1,\ldots,2k\}$.
Next for every $e \in E(H)$, we add arcs to $D$ such that $D\langle Z_e\rangle$ is a $k$-strong tournament. 
Further, for every $e \in E(H)$, $i\in [k]$ and $j\in [k-1]$, we add an arc from $z_e^{i}$ to $w_{j}$ and for every $e \in E(H)$, $i\in\{k+1,\ldots,2k\}$ and $j\in [k-1]$, we add an arc from  $w_{j}$ to $z_e^{i}$. 
Finally, for every $e\in E(H)$ and $v \in e$, we add the arc $v z_e^{1}$. This finishes the description of $D$. 
Observe that $|V(D)|=|V(H)|+(2k+1)|E(H)|+2k+1$. 
We now show that $\sinv_k(D)=\cc(H)$.
Note that this implies that $\sinv_k(D)$ is finite and thus $D$ is $k$-strengthenable.

\smallskip

We first show that $\sinv_k(D)\leq \cc(H)$. Let $(X_1,\ldots,X_t)$ be an optimal cut cover of $H$. 
Now for every $e\in E(H)$, there is a smallest $\alpha(e)$ such that $e \cap X_{\alpha(e)}$ and $e \setminus  X_{\alpha(e)}$ are nonempty. 
Now for every $i\in[t]$, let $X'_i=X_i\cup\{z_e^1:\alpha(e)=i\}$. Let $D'$ be the digraph obtained from $D$ by inverting $X'_1,\ldots,X'_t$. 
We will show that $D'$ is $k$-strong. Let $Y\subseteq V(D)$ with $|Y|\leq k-1$. We need to show that $D'-Y$ is strongly connected. 
As $D\langle W\rangle$ is $k$-strong, we obtain that $W\setminus Y$ is contained in a single strongly connected component $S$ of $D'-Y$. Next observe that every $v \in V(H)\setminus Y$ has an in-neighbour and an out-neighbour in $S$, so $V(H)\setminus Y \subseteq S$. 
Now consider some $e\in E(H)$. As $D'\langle Z_e\rangle=D\langle Z_e\rangle$ is $k$-strong, we obtain that $Z_e\setminus Y$ is contained in a single strongly connected component of $D'-Y$. 
If $Y\neq \{w_1,\ldots,w_{k-1}\}$, then there is at least one arc from $Z_e\setminus Y$ to $W \setminus Y$ and there is at least one arc from $W \setminus Y$ to $Z_e \setminus Y$, so $Z_e\setminus Y \subseteq S$. 
Finally, if $Y=\{w_1, \dots, w_{k-1}\}$, observe that, by construction, $D'$ contains an arc from $V(H)$ to $z_e^1$ and an arc from $z_e^1$ to $V(H)$. 
Again, we obtain $Z_e\setminus Y \subseteq S$. This yields that $D'$ is $k$-strong. We hence have $\sinv_k(D)\leq t=\cc(H)$.

\smallskip

We now show that $\cc(H)\leq\sinv_k(D)$. 
Let $(X_1,\ldots, X_t)$ be a minimum $k$-strengthening family of $D$.
For $i=1,\ldots,t$, let $X_i'=X_i \cap V(H)$.
We claim that $(X'_1, \dots, X'_t)$ is a cut cover of $H$.
In other words, for every $e \in E(H)$, there is some $i \in \{1,\ldots,t\}$ such that $e \cap X_{i}$ and $e \setminus  X_{i}$ are nonempty.

Suppose otherwise, that is for every $i\in [t]$ either $e \subseteq X_i$ or $e \cap X_i=\emptyset$ holds.
Thus all the arcs corresponding to edges in $\partial_{D}(Z_e)\cap \partial_{D}(z_e^{1})$ are inverted or none of them. This yields that in $D'- \{w_1,\ldots,w_{k-1}\}$,  there is either no arc entering $Z_e$ or no arc leaving $Z_e$. Hence $D'-\{w_1,\ldots,w_{k-1}\}$ is not strongly connected, a contradiction.
Therefore, we obtain that $(X_1',\ldots,X_t')$ is a cut cover for $H$. Thus $\cc(H)\leq t=\sinv_k(D)$, and we conclude that $\cc(H) = \sinv_k(D)$.
\end{proof}

Let $k$ and $t$ be two positive integers.
Lemma~\ref{lem:reduc3} and the NP-hardness of deciding whether a given hypergraph $H$ satisfies $\cc(H)\leq t$ for all $t\geq 1$ directly imply that deciding whether a given oriented graph $D$ satisfies $\sinv_k(D) \leq t$ is NP-hard.
This problem is clearly in NP because a family $\mathcal{X}$ such that $\Inv(D;\mathcal{X})$ is $k$-strong is a polynomial-size certificate that $\sinv'_k(D) \leq |\mathcal{X}|$.  
Hence it is NP-complete, which proves Theorem~\ref{verhard}.

Moreover Proposition~\ref{approcc} and Lemma~\ref{lem:reduc3}  imply Theorem~\ref{thm:inapprox-vertex}.

Lemma~\ref{lem:reduc3} and the NP-hardness of deciding whether a given hypergraph $H$ satisfies $\cc(H)\leq t$ for all $t\geq 1$ directly imply the following.
\begin{corollary}\label{rsedgzj}
Let $t$ and $k$ be positive integers.
Deciding whether $\sinv_k(D) \leq t$  for a given oriented graph $D$ is NP-complete.
\end{corollary}

\section{Bounds on \texorpdfstring{$M'_k$}{M'k} and \texorpdfstring{$M_k$}{Mk}}\label{sec:M}

This section is dedicated to giving upper and lower bounds on $M_k$ and $M_k'$, in particular proving Theorems~\ref{thm:exttournoi} and~\ref{thm:upperM'}. First, in Section \ref{sec:52}, we completely determine $M_k$ and $M_k'$ for some small values of $k$. Next, in Section \ref{sec:51}, we prove Theorem \ref{thm:m(2k+1)}, which is the lower bound in Theorem \ref{thm:exttournoi} and a slight extension of this result.   In Section \ref{sec:53}, we prove Theorem \ref{thm:M<2k}, which is the upper bound in Theorem \ref{thm:exttournoi}. We further give a slight improvement of this result for tournaments on exactly $2k+1$ vertices. Finally, in Section \ref{sec:54}, we give the somewhat involved, probabilistic proof of Theorem \ref{thm:upperM'}.

\subsection{Values of \texorpdfstring{$M'_1$}{M'1}, \texorpdfstring{$M_1$}{M1}, \texorpdfstring{$M'_2$}{M'2} and \texorpdfstring{$M_2$}{M2}}\label{sec:52}
We here provide the exact values of $M_i$ and $M_i'$ for $i \in \{1,2\}$.
\begin{proposition}\label{prop:M1}
Let $T$ be a tournament of order $n \geq 3$.
We have $\sinv_1(T) = \sinv'_1(T) =0$ if $T$ is strong and $\sinv_1(T) = \sinv'_1(T) =1$ otherwise.
In particular,  $m_1(n) = m'_1(n) =1$ for all $n\geq 3$ and $M_1=M'_1=1$.
\end{proposition}
\begin{proof}
Trivially, if $T$ is strong, then $\sinv_1(T)=0$. 
If $T$ is not strong, then $\sinv_1(T) \geq 1$. Now consider a hamiltonian path of $T$.
Such a path does exist by Redei's Theorem (see e.g. Theorem~1.4.2 in \cite{bang2009}). Let $x$ be its initial vertex and $y$ its terminal vertex.
As $T$ is not strong, we have $xy \in A(T)$, and hence inverting $\{x,y\}$ yields a tournament with a directed hamiltonian cycle because $V(T) \setminus\{x,y\} \neq \emptyset$.
It follows that this tournament is strong, and so $\sinv_1(T) = 1$.
\end{proof}


For every integer $n$, let $TT_n$ be the unique (up to isomorphism) acyclic tournament of order $n$.
Let $\vec{C}_3$ be the directed triangle, and let $S_4$ be the unique (up to isomorphism) strong tournament of order $4$.
Its vertex set is $\{a,b,c,d\}$ and its arc set is $\{ab, bc, cd, da, ca, db\}$.

\begin{proposition}\label{prop:M2}
$M_2=M'_2=2$.
\end{proposition}
\begin{proof}
The rotative tournament $R_5$ of order $5$ is the only 2-arc-strong tournament of order $5$.
As observed in \cite{BBBP10}, we have $\inv(R_5)=2$, so $\sinv'_2(TT_5) = 2$.
Hence $M_2 \geq M'_2\geq 2$.

\medskip

Let us now prove that $M_2\leq 2$.
We shall prove by induction on $n$ that every tournament $T$ of order at least $5$ satisfies $\sinv_2(T) \leq 2$.

Assume first that $T$ is a tournament of order $5$.
If $T$ is strong, then, by Camion's theorem~\cite{camion1959}, it has a hamiltonian cycle
$v_1v_2v_3v_4v_5v_1$. Let $A^+=A(T)\cap \{v_1v_3, v_2v_4, v_3v_5, v_4v_1, v_5v_2\}$ and $A^-=A(T)\cap \{v_3v_1, v_4v_2, v_5v_3, v_1v_4, v_2v_5\}$.
We have $|A^+| + |A^-|=5$, so one of the two sets $A^+, A^-$ has at most two arcs. Reversing the arcs of this set, one after another, yields the 2-strong tournament $R_5$.

Assume now that $T$ is not strong. 
Then by possibly replacing $T$ by its {\bf converse}, that is the tournament $\Inv(D,V(D))$, one of the following property holds.
\begin{itemize}
\item $T = TT_5$ with hamiltonian path $v_1v_2v_3v_4v_5$. Then inverting $\{v_1, v_2, v_4, v_5\}$ and $\{v_1, v_5\}$ yields $R_5$.

\item $S_4\Ra \{x\}$. Then inverting $\{c, d, x\}$ and $\{c, d\}$ yields $R_5$.

\item $(\{x\} \Ra \vec{C}_3) \Ra \{y\}$ with $ \vec{C}_3 = abca$. Then inverting $\{a,x,y\}$ yields $R_5$.

\item $(\{x\} \Ra \{y\}) \Ra \vec{C}_3 $ with $ \vec{C}_3 = abca$.
Then inverting $\{a,b,c,y\}$ and $\{a,x,y\}$ yields $R_5$.

\end{itemize}
As $\sinv_2(\Inv(D,V(D)))=\sinv_2(D)$ for digraph $D$, the statement follows.
Assume now that $T$ has at least $6$ vertices.

Assume first $T$ has a vertex $v$ such that $\min\{d^+(v),d^-(v)\}\geq 2$.
By the induction hypothesis, $\sinv_2(T-v) \leq 2$, so there is a family ${\cal X}$ of at most two subsets of $V(T-v)$ such that
$\Inv(T-v; {\cal X})$ is $2$-strong.
Now $\Inv(T; {\cal X})-v = \Inv(T-v; {\cal X})$ and $\min\{d^+(v),d^-(v)\}\geq 2$. Thus, by Lemma~\ref{lem:kstrong+},
$\Inv(T; {\cal X})$ is $2$-strong, and it follows that $\sinv_2(T) \leq |{\cal X}| \leq 2$.

Assume now that every vertex has either in-degree at most $1$ or out-degree at most $1$. Then necessarily $T$ must be the tournament $\vec{C}_3\Ra \vec{C}_3$.
Let $V(T) = \{a,b,c,d,e,f\}$ with $\{a,b,c\} \Ra \{d,e,f\}$.
Then inverting $\{a,b,c,d\}$ and $\{a, d,e,f\}$ transforms $T$ into a $2$-strong tournament.
\end{proof}

\subsection{Lower bound on \texorpdfstring{$m'_k(2k+1)$}{mk(2k+1)}}\label{sec:51}

We shall first show the lower bound $m'_k(2k+1) = \Omega (\log k)$. We need the following result.


\begin{theorem}[McKay~\cite{McK90}]\label{count}
Let $n$ be an odd integer.
The number of labelled eulerian tournaments on $n$ vertices is  $$\left(\frac{2^{n+1}}{\pi n}\right)^{\frac{n-1}{2}}\sqrt{\frac{n}{\e}} (1 +o(1)).$$
\end{theorem}
We are now ready to prove the lower bound in Theorem~\ref{thm:exttournoi}.
\begin{theorem}\label{thm:m(2k+1)}
For every sufficiently large $k$, $m'_k(2k+1)\geq \frac{1}{2}\log (2k+1)$.
\end{theorem}
\begin{proof}
    By Theorem ~\ref{count}, we may choose an integer $k_0$ such that for all $k \geq k_0$ and for $n=2k+1$, the number of labelled eulerian tournaments on $n$ vertices is at most $\left(\frac{2^{n+1}}{\pi n}\right)^{\frac{n-1}{2}}\sqrt{n}$ and $\frac{n-1}{2}(\log \pi -1)>\log n$. Now fix some $k\geq k_0$ and let $n=2k+1$.  By Proposition~\ref{eul}, all the $k$-arc-strong tournaments on $n$ vertices are eulerian. 
    For two labelled tournaments $T,T'$ on the same vertex set, we say that $T'$ is {\bf reachable} from $T$ by $t$ inversions for some positive integer $t$, if there is a family of sets $(X_1,\ldots,X_t)$ such that $\Inv(T;X_1,\ldots,X_t)=T'$. Observe that there are $2^n$ possibilities to choose $X_i$ for $i=1,\ldots,t$, hence the number of tournaments reachable from $T$ by $t$ inversions is at most $(2^n)^t=2^{nt}$. Therefore,
    the number of $n$-vertex labelled tournaments that are reachable by $t$ inversions from an eulerian  one is at most 
    \begin{eqnarray*}
    2^{nt}\left(\frac{2^{n+1}}{\pi n}\right)^{\frac{n-1}{2}}\sqrt{n}
    & = &2^{\binom{n}{2}}\cdot 2^{nt} \left(\frac{2}{\pi n}\right)^{\frac{n-1}{2}}\sqrt{n} \\
    & = &2^{\binom{n}{2}}\cdot 2^{nt +\frac{n-1}{2}(1 -\log(\pi)-\log n)+\frac{1}{2}\log n}\\
    & < &2^{\binom{n}{2}}\cdot 2^{nt-\frac{1}{2} n \log n}
    \end{eqnarray*}
    If $t \leq \frac{1}{2}\log n$, then $nt - \frac{1}{2} n \log n \leq 1$, so the number of such tournaments is fewer than $2^{\binom{n}{2}}$. It follows that there is at least one tournament $T^*$ on $2k+1$ vertices which cannot be reached from an eulerian tournament by at most $t$ inversions. Thus $\sinv'_k(T^*)>t$.
\end{proof}

Theorem~\ref{thm:m(2k+1)} can be slightly generalised to tournaments on $2k+c$ vertices for small integers $c$. To do so, we will need the following generalisation of Theorem~\ref{count}.

For a positive integer $n$ and a collection of integers $\alpha_1,\dots ,\alpha_n$, we denote by
$NT(n;\alpha_1, \dots,\alpha_n)$ the number of labelled tournaments
on $[n]$ in which the vertex $i$ satisfies $d^+(i)-d^-(i)=\alpha_i$.

\begin{theorem}[Spencer~\cite{Spencer74} and McKay~\cite{McK90}]\label{thm:spencer}
Let $c$ be an integer and let $\epsilon > 0$.
There is an integer $n_0$ such that for every $n \geq n_0$,
for every $\alpha_1, \dots, \alpha_n \in [-c,c]$, we have
\[
NT(n,\alpha_1, \dots,\alpha_n) \leq
\left(\frac{2^{n+1}}{\pi n}\right)^{(n-1)/2} \sqrt{\frac{n}{\e}} \exp\left(-\frac{1-\epsilon}{2n}\sum_{i=1}^n\alpha_i^2\right).
\]
\end{theorem}




\begin{corollary}\label{corollary:number_tournaments_high_degree}
For all integers $k$ and $c$, the number of labelled tournaments $T$ on $n=2k+c$ vertices
such that every vertex has in- and out-degree at least $k$ is at most
$2^{\log(2c)n}\left(\frac{2^{n+1}}{\pi n}\right)^{(n-1)/2} \sqrt{\frac{n}{\e}}$ for $k$ large enough.
\end{corollary}

\begin{proof}
By Theorem~\ref{thm:spencer} for $\epsilon = 1$,
for $k$ large enough,
the number of labelled tournaments $T$ on $n=2k+c$ vertices
such that every vertex has in- and out-degree at least $k$ is at most

\[
\left(\frac{2^{n+1}}{\pi n}\right)^{(n-1)/2} \sqrt{\frac{n}{\e}}
\sum_{(\alpha_1, \dots,\alpha_n) \in \{-c+1, \dots, c-1\}^n} 1,
\]
which is
\[
\left(\frac{2^{n+1}}{\pi n}\right)^{(n-1)/2} \sqrt{\frac{n}{\e}}
(2c-1)^n
\]
and the corollary follows from the fact that $(2c-1)^n \leq 2^{\log(2c)n}$.
\end{proof}


\begin{theorem}\label{thm:2k+c}
    For every fixed positive integer $c$, and for every positive integer $k$ large enough compared to $c$, there exists a tournament
    $T$ on $2k+c$ vertices such that $\sinv'_k(T) > \frac{1}{2}\log (2k+c)-\log(2c)$.
    In particular, $m'_k(2k+c)$ is unbounded for every fixed $c$.
\end{theorem}

\begin{proof}
By Corollary~\ref{corollary:number_tournaments_high_degree},
the number of labelled tournaments $T$ on $n=2k+c$ vertices with $\sinv'_k(T) \leq t$
is for $k$ large enough at most 
\[
2^{nt} 2^{\log(2c)n}\left(\frac{2^{n+1}}{\pi n}\right)^{(n-1)/2} \sqrt{\frac{n}{\e}}
< 2^{\binom{n}{2}} 2^{nt-\frac{1}{2}n\log n+\log(2c)n}
\]
which is at most $2^{\binom{n}{2}}$ if $t \leq \frac{1}{2}\log n -\log(2c)$.
Hence there exists a tournament $T^*$ with $\sinv'_k(T^*) > \frac{1}{2}\log (2k+c)-\log(2c)$.
\end{proof}

\subsection{Upper bounds on \texorpdfstring{$M_k$}{Mk} and transitive tournaments}\label{sec:53}
In this section, we give some upper bounds on $M_k$ for all positive integers $k$. Most notably, we prove the upper bound in Theorem~\ref{thm:exttournoi}.

\begin{theorem}\label{thm:M<2k}
$M_k \leq 2k$.
\end{theorem}

\begin{proof}
Let $T$ be a tournament with $V(T)=\{v_1,\ldots,v_n\}$ with $n \geq 2k+1$. Further, let $T'$ be a $k$-strong tournament on $\{v_1,\ldots,v_{2k+1}\}$. We now define sets $X_1,\ldots,X_{2k}$. 
Suppose that the sets $X_1,\ldots,X_{i-1}$ have already been created and let $T_{i-1}$ be the graph obtained from $T$ by inverting $X_1,\ldots,X_{i-1}$.
Now let $X_i = \{v_i\} \cup A_i \cup B_i$, where
$A_i$ is the set of vertices $v_j$ with $j \in \{i+1,\ldots,2k+1\}$ for which the edge $v_iv_j$ has a different orientation in $T'$ and $T_{i-1}$,
and $B_i$ is, when $i \leq k$ (resp. $i \geq k+1$), the set of vertices $v_j$ with $j \geq 2k+2$ for which $T_{i-1}$ contains the arc $v_iv_j$ (resp. $v_jv_i$).

We still need to show that $T_{2k}$ is $k$-strong. Observe that $T_{2k}\langle\{v_1,\ldots,v_{2k+1}\}\rangle=T'$ which is $k$-strong by assumption. Moreover, for any $j\geq 2k+2$, $T_{2k}$ contains the arcs $v_jv_i$ for $i=1,\ldots,k$ and the arcs $v_iv_j$ for $i=k+1,\ldots,2k$. Hence, by Lemma~\ref{lem:kstrong+}, $T_{2k}$ is $k$-strong.
\end{proof}

Theorems~\ref{thm:m(2k+1)} and~\ref{thm:M<2k} directly imply Theorem~\ref{thm:exttournoi}.

It is tempting to improve the upper bound in Theorem~\ref{thm:exttournoi}. While we are not able to provide an improvement for the general case, we show in the following that a small improvement can be achieved in a seemingly critical case, namely when the size of the tournament is exactly $2k+1$. We first need the following result.

\begin{theorem}\label{thm:TTn}
    Let $n,k$ be integers with $n \geq 2k+1$.
    \[
    \sinv_k(TT_n) = \sinv'_k(TT_n) =
    \left\{
    \begin{array}{l l}
        2 & \text{ if~~} 2k+1 \leq n < 3k \\
        1 & \text{ if~~} 3k \leq n. \\
    \end{array}
    \right.
    \]
\end{theorem}

\begin{proof}
We first prove that $\sinv_k(n)\leq 2$.    Let $(v_1, \dots, v_n)$ be the unique ordering of $V(TT_n)$ such that
    $v_i v_j \in A(TT_n)$ for every $i<j$.
    Let $X_0 = \{v_i \mid i \text{ even}\}$ and $X_1 = \{v_i \mid i \text{ odd}\}$.
    We claim that $T' = \Inv(TT_n; X_0,X_1)$ is $k$-strong.
    Consider  a set $Y$ of at most $k-1$ vertices of $T'$.
    Let $i_0$ (resp. $j_0$) be the smallest (resp. largest) even integer $\ell$ such that $v_\ell \not\in Y$, and let
    $i_1$ (resp. $j_1$) be the smallest (resp. largest) odd integer $\ell$ such that $v_\ell \not\in Y$.
    Observe that $j_0 > i_1$ and $j_1 > i_0$ since $|Y|<k$ and $n \geq 2k+1$.
    In particular, $v_{i_1}v_{j_0}, v_{i_0}v_{j_1} \in A(T')$.
    Since $T'\langle X_0 \rangle-Y$ (resp. $T'\langle X_1\rangle-Y$) is a transitive tournament with source $v_{j_0}$ and sink $v_{i_0}$ (resp. source $v_{j_1}$ and sink $v_{i_1}$), there are hamiltonian paths $P_0$ from $v_{j_0}$ to $v_{i_0}$ in $T'\langle X_0 \rangle-Y$ and $P_1$ from $v_{j_1}$ to $v_{i_1}$ in $T\langle X_1 \rangle-Y$.
    Finally $P_0v_{i_0}v_{j_1}P_1v_{i_1}v_{j_0}$ is an hamiltonian directed cycle of $T'-Y$, and so $T'-Y$ is strong.
    Hence $T'$ is $k$-strong, and $\sinv_k(TT_n) \leq 2$.

    Let us now prove that $\sinv'_k(TT_n) \geq 2$ if $n \leq 3k-1$.
    Suppose for a contradiction that there exists $X \subseteq V(TT_n)$ such that $T'=\Inv(TT_n;X)$ is $k$-arc-strong.
    Then, for every $i \leq k$, since $v_i$ has in-degree $i-1$ in $TT_n$,  we have $v_i \in X$.
    Similarly, for every $i \geq n-k+1$, since $v_i$ has out-degree $n-i$ in $TT_n$, we have $v_i \in X$.
    But then every out-neighbour of $v_1$ in $T'$ is in $\{v_{k+1}, \dots, v_{n-k}\}$ and so $d^+_{T'}(v_1) \leq n-2k < k$.
    Hence $T'$ is not $k$-arc-strong, a contradiction.

    Finally, let us show that $\sinv_k(TT_n) \leq 1$ if $n \geq 3k$.
    Let $(v_1, \dots, v_n)$ be the unique ordering of $V(TT_n)$ such that
    $v_i v_j \in A(TT_n)$ for every $i<j$.
    Let $A = \{v_1, \dots, v_k\}, B= \{v_{k+1},\dots, v_{2k}\}$ and $C = \{v_{2k+1}, \dots, v_n\}$.
    In $T' = \Inv(TT_n; A\cup C)$, we have $A \Rightarrow B, B \Rightarrow C$ and $C \Rightarrow A$. Since $|A|,|B|,|C| \geq k$, $T'$ is $k$-strong.
\end{proof}

As mentioned in the introduction, it was proven independently by Alon et al. \cite{APSSW} and Aubian et al. \cite{inversion} that $\inv(n) \leq n - \lceil \log (n+1) \rceil$. Hence , every tournament of order $2k+1$ can be made acyclic in at most $2k - \lceil \log (k+1) \rceil$ inversions, and $k$-strong in two more inversions by Theorem~\ref{thm:TTn}. Thus we have the following.
\begin{corollary}
$m_k(2k+1) \leq 2k - \lceil \log (k+1) \rceil +2$.  
\end{corollary}

\subsection{Upper bound on \texorpdfstring{$M'_k$}{M'k}}\label{sec:54}

This section is dedicated to proving Theorem~\ref{thm:upperM'}. First, we give some auxiliary results in Section~\ref{prelm'}. Next, in Section~\ref{small}, we show the result when restricting to tournaments of size exactly $2k+1$. We finally generalize this to bigger tournaments in Section~\ref{big}.

\subsubsection{Preliminaries}\label{prelm'}
We here collect some auxiliary results we need for the proof of Theorem~\ref{thm:upperM'}.
We first need a basic result combining probability and linear algebra.

\begin{proposition}\label{linalg}
    Let $A\in \mathbb{F}_2^{q \times q'}$ be a matrix whose rank is $q$ for some integers $q,q'$ with $q \leq q'$. Further, let $\vec{v}\in \mathbb{F}_2^q$ be a fixed vector and let another vector $\vec{w}\in \mathbb{F}_2^{q'}$ be drawn uniformly at random. Then $\Pr(A\vec{w}=\vec{v})=2^{-q}$.
\end{proposition}
The simple proof of the following result is similar to the one of Theorem~\ref{thm:M<2k} (see~\cite[Lemma~2.1]{havet2024diameter}).
\begin{proposition}\label{fuzzu}
    Let $T_1,T_2$ be tournaments with $V(T_1)=V(T_2)$. Then there is a collection of $|V(T_1)|-1$ subsets $X_1,\ldots,X_{|V(T_1)|-1}$ of $V(T_1)$ such that $\Inv(T_1;X_1,\ldots,X_{|V(T_1)|-1})=T_2$.
\end{proposition}

We are now ready to give the last preliminary result.
\begin{proposition}\label{extend}
    Let $T$ be a tournament on $2k+1$ vertices and $X \subseteq V(T)$ with $|X|\leq \frac{2}{3}k$ such that $d_T^-(x)=d_T^+(x)=k$ for all $x \in X$. Then there exists a $k$-arc-strong tournament $T'$ on $V(T)$ such that all the edges in $E(\UG(T)\langle X \rangle) \cup \delta_{\UG(T)}(X)$ have the same orientation in $T$ and $T'$.
\end{proposition}

\begin{proof}
    Let $H$ be the mixed graph which is obtained from $\UG(T)$ by giving all edges in $E(\UG(T)\langle X \rangle) \cup \delta_{\UG(T)}(X)$ the orientation they have in $T$. Now consider some $S \subseteq V(T)$.
    Let $x \in S\cap X$. Since $d_H^+(x)=d_H^-(x)$, we have
    \begin{align*}
        d_H^+(S\setminus \{x\})-d_H^-(S\setminus \{x\})
        &=\left(d_H^+(S)-d_H^+(x,V(H)\setminus S)+d_H^+(S\setminus \{x\},x)\right)\\
        &\qquad -\left(d_H^-(S)-d_H^-(x,V(H)\setminus S)+d_H^-(S\setminus \{x\},x)\right)\\
        &=d_H^+(S)-d_H^+(S)\\
        &\qquad +d_H^+(S\setminus \{x\},x)+d_H^+(V(H)\setminus S,x)\\
        &\qquad -d_H^+(x,V(H)\setminus S)-d_H^+(x,S\setminus \{x\}))\\
        &=d_H^+(S)-d_H^+(S)+d_H^-(x)-d_H^+(x)\\
        &=d_H^+(S)-d_H^-(S).
    \end{align*}
    Repeatedly applying this argument, we obtain $d_H^+(S)-d_H^-(S)=d_H^+(S \setminus X)-d_H^-(S \setminus X) \leq d_H^+(S \setminus X) \leq |S \setminus X| \cdot |X|$ since every arc in $H$ has an extremity in $X$. Symmetrically, we obtain $d_H^+(S)-d_H^-(S)=d_H^-(V(H)\setminus S)-
    d_H^+(V(H)\setminus S)=d_H^-(V(H)\setminus (S\cup X))-
    d_H^+(V(H)\setminus (S\cup X))\leq |V(H)\setminus (S\cup X)| \cdot |X|$.
    
    As $|X|\leq \frac{2}{3}k\leq \frac{1}{2}|V(H)\setminus X| = \frac{1}{2}(|S\setminus X| + |V(H)\setminus(S \cup X)|)$, we obtain
    \begin{align*}
        d_H^+(S)-d_H^-(S)&\leq \min \{|S \setminus X|,\ |V(H)\setminus (S\cup X)|\} \cdot |X|\\
        & \leq \min \{|S \setminus X|,\ |V(H)\setminus (S\cup X)|\} \cdot \tfrac{1}{2}(|S\setminus X| + |V(H)\setminus(S \cup X)|)\\
        & \leq \min \{|S \setminus X|,\ |V(H)\setminus (S\cup X)|\} \cdot \max \{|S \setminus X|,\ |V(H)\setminus (S\cup X)|\}\\
        &=|S \setminus X| \cdot |V(H)\setminus (S\cup X)|\\
        &= d_H(S).
    \end{align*}
    %
    Hence, by Proposition~\ref{prop:eul-or}, there is an eulerian orientation $T'$ of $T$ such that all the edges in $\UG(T)\langle X \rangle \cup \delta_{\UG(T)}(X)$ have the same orientation in $T$ and $T'$. As $|V(T)|=2k+1$, the tournament $T'$ is $k$-arc-strong by Proposition~\ref{eul}.
\end{proof}

\subsubsection{Main proof for tournaments of size \texorpdfstring{$2k+1$}{2k+1}}\label{small}
We here give the proof of Theorem~\ref{thm:upperM'} for tournaments of size $2k+1$. More precisely, we prove the following statement.
\begin{theorem}\label{mkstrich}
For every $\epsilon >0$, there is an integer $k_0$ such that for every $k \geq k_0$ and for every tournament $T$ on $2k+1$ vertices, we have $\sinv'_k(T)\leq (\frac{4}{3}+\epsilon)k$.
\end{theorem}

\begin{proof}
Let $\epsilon >0$. We may assume $\epsilon \leq 1/3$.
We choose $k_0$ large enough so that for all $k \geq k_0$, the  three following inequalities hold:

\begin{itemize}

\item $\lceil \log^2 k\rceil + \lceil\frac{1}{4}\log k\rceil +  \frac{2k}{3\lceil\frac{1}{4}\log k\rceil} \leq \epsilon k$

\item $16 \log (k) \cdot \lceil \frac{1}{4}\log k\rceil \leq \sqrt{k}$.

\item $10 \log k \leq k^{\frac{1}{8}}$, 


\end{itemize}

In particular, the last inequality is wrong for $k=2^{12}$ which implies $k_0\ge 2^{12}$. Thus the following inequalities also hold for all $k\geq k_0$: 
$\lceil\frac{1}{4}\log k\rceil\leq \frac{1}{3} k$; 
$\lceil\frac{1}{4}\log k\rceil\leq \sqrt{k}$; 
$k\geq 13$; 
$ 3 \lceil \frac{1}{4} \log k\rceil \leq \log k$.

\medskip

Let $T$ be a tournament on $2k+1$ vertices for some $k \geq k_0$. Let $A \subseteq V(T)$ be an arbitrary subset of $V(T)$ with $|A|=k$ and $B=V(T)\setminus A$. 
For some $b \in B$ and a tournament $\tilde{T}$ on $V(T)$, we denote $|k-d_{\tilde{T}}^-(b)|$ by {\boldmath $\defect_{\tilde{T}}(b)$}, the {\bf defect of $b$ in $\Tilde{T}$}. Let $q=\lceil\frac{1}{4}\log k\rceil$ and $q^* = \lceil \log^2 k\rceil$. We now choose $q^*$ sets $X_1,\ldots,X_{q^*}\subseteq V(T)$ independently and uniformly at random. 
Let $T'=\Inv(T;X_1,\ldots,X_{q^*})$. For some $b \in B$, we denote by $\vec{b}$ the vector in $\mathbb{F}_2^{q^*}$ whose $i$-th entry is 1 if $b \in X_i$ and 0 otherwise for $i\in[q^*]$.
Observe that by the choice of $X_1,\ldots,X_{q^*}$, the vectors $\{\vec{b} \mid b \in B\}$ follow a uniform distribution and are mutually independent.

We now define a list of possible events:

\begin{enumerate}[label=\Alph*]
    \item[$E_1$]: There is some $b \in B$ with $\defect_{T'}(b)\geq 2 \sqrt{k} \log k $.
    \item[$E_2$]: There are some $b_1,\ldots,b_q \in B$ such that $\vec{b}_1,\ldots,\vec{b}_q$ are linearly dependent.
    \item[$E_3$]: There are some $b_1,\ldots,b_q \in B$ such that $\vec{b}_1,\ldots,\vec{b}_q$ are linearly independent and a sequence $(\star_1,\ldots,\star_q)\in \{+,-\}^q$ such that $|\bigcap_{i=1}^q N_{T'}^{\star_i}(b_i)\cap A|\leq 5 \sqrt{k} \log k $.
\end{enumerate}

\begin{lemma}\label{erfolg}
    If none of $E_1$, $E_2$ and $E_3$ occur, then $\sinv_k'(T)\leq (\frac{4}{3}+\epsilon)k$.
\end{lemma}

\begin{proof}
Let $(B_1,\ldots,B_t)$ be a maximal collection of disjoint subsets of $B$ with the following properties:
\begin{itemize}
    \item $|B_i|=q$ for $i=1,\ldots,t$,
    \item $\defect_{T'}(b)\equiv \defect_{T'}(b') \mod 2$ for all $b,b' \in B_i$ for $i=1,\ldots,t$,
    \item $tq \leq \frac{2}{3}k$.
\end{itemize}

\begin{claim}\label{bgross}
    $|\bigcup_{i=1}^tB_i|\geq \frac{2}{3}k-q$.
\end{claim}
\begin{proofclaim}
    Suppose otherwise. Observe that, as $\lceil\frac{1}{4}\log k\rceil\leq \frac{1}{3} k$, we have $|B\setminus (\bigcup_{i=1}^tB_i)|=|B|-|\bigcup_{i=1}^tB_i|\geq (k+1)-(\frac{2}{3}k-q)\geq \frac{1}{3}k+q \geq 2q$. 
    By the Pigeonhole Principle, there exists a set $B_0\subseteq B\setminus (\bigcup_{i=1}^tB_i)$ with $|B_0|=q$ such that $\defect_{T'}(b)\equiv \defect_{T'}(b') \mod 2$ for all $b,b' \in B_0$. Hence $B_0$ can be added to $\mathcal{B}$ without violating any of the above conditions, a contradiction to the maximality of $(B_1,\ldots,B_t)$.
\end{proofclaim}

\begin{claim} Let $i\in [t]$.
    There is a set $Y_i\subseteq B_i \cup A$ such that $\defect_{\Inv(T';Y_i)}(b)=0$ for all $b \in B_i$.
\end{claim}
\begin{proofclaim}
    Let $p=\max\{\defect_{\Inv(T';B_i)}(b) \mid b \in B_i\}$. Observe that $p\leq  \max\{\defect_{T'}(b)+q-1 \mid b \in B_i\}$. Hence, as $E_1$ does not occur and $q=\lceil\frac{1}{4}\log k\rceil\leq \sqrt{k} \log k $, we have $p\leq 2 \sqrt{k} \log k + \sqrt{k} \log k = 3 \sqrt{k} \log k $. We shall now build a sequence of sets $Y_i^0,\ldots,Y_i^p$ with the following properties:
    
    \begin{itemize}
        \item $Y_i^0=B_i$,
        \item $Y_i^{j+1}$ is obtained from $Y_i^j$ by adding one vertex from $A\setminus Y_i^j$ for $j=0,\ldots,p-1$,
        \item $\max\{\defect_{\Inv(T';Y_i^j)}(b) \mid b \in B_i\}=p-j$ for $j=0,\ldots,p$.
    \end{itemize}
    The statement then follows for $Y_i=Y_i^p$. 
    
    Suppose that we have already created the sets $Y_i^0,\ldots,Y_i^j$ for some $j$ with $0\leq j \leq p-1$ and now want to create $Y_i^{j+1}$. 
    Let $(b_1,\ldots,b_q)$ be an arbitrary ordering of the vertices in $B_i$. 
    For every $\mu \in [q]$, if  $\max\{\defect_{\Inv(T';Y_i^j)}(b) \mid b \in B_i\}=d_{\Inv(T';Y_i^j)}^-(b_\mu)-k$, let $\star_\mu=-$, if $\max\{\defect_{\Inv(T';Y_i^j)}(b) \mid b \in B_i\}=k-d_{\Inv(T';Y_i^j)}^-(b_\mu)$, let $\star_\mu=+$, otherwise choose $\star_\mu \in \{+,-\}$ arbitrarily. 
    As $E_2$ does not occur, $\vec{b}_1,\ldots,\vec{b}_q$ are linearly independent. 
    Moreover, as $E_3$ does not occur, $|\bigcap^q_{\mu=1} N^{\star_\mu}_{T'}(b_\mu)\cap A|\geq 5 \sqrt{k} \log k$. 
    As $p \leq 3 \sqrt{k} \log k$, and $q \leq \sqrt{k} \log k $, we have $|Y_i^j|\leq p+q \leq 4 \sqrt{k} \log k$. 
    Thus $(\bigcap_{\mu=1}^q N_{T'}^{\star_\mu}(b_\mu)\cap A)\setminus Y_i^j\neq \emptyset$. 
    We can hence choose some $y \in (\bigcap_{\mu=1}^q N^{\star_\mu}_{T'}(b_\mu)\cap A)\setminus Y_i^j$ and define $Y_i^{j+1}=Y_i^j\cup \{y\}$. 
    By definition, $Y_i^{j+1}$ is obtained from $Y_i^j$ by adding one vertex in $A\setminus Y_i^j$. 
    In order to see that the last property holds, let $\mu \in [q]$. 
    If $\max\{\defect_{\Inv(T';Y_i^j)}(b)\mid b \in B_i\}=d_{\Inv(T';Y_i^j)}^-(b_\mu)-k$, we have $\star_\mu=-$ and hence $y \in N^-_{\Inv(T';Y_i^j)}(b_\mu)$. 
    As $d_{\Inv(T';Y_i^j)}^-(b_\mu)-k >0$, we obtain $\defect_{\Inv(T';Y_i^{j+1})}(b_\mu)=d_{\Inv(T';Y_i^{j+1})}^-(b_\mu)-k=d_{\Inv(T';Y_i^j)}^-(b_\mu)-k-1=p-j-1=p-(j+1)$. 
    Similarly, if $\max\{\defect_{\Inv(T';Y_i^j)}(b)\mid b \in B_i\}=k-d_{\Inv(T';Y_i^j)}^-(b_\mu)$, then $\defect_{\Inv(T';Y_i^{j+1})}(b)=p-(j+1).$ 
    Otherwise, let $b' \in B$ with $\defect_{\Inv(T';Y_i^j)}(b')=p-j$.
    As $\defect_{T'}(b_\mu) \equiv \defect_{T'}(b') \mod 2$, we have $\defect_{\Inv(T';Y_i^j)}(b_\mu)\equiv \defect_{T'}(b_\mu)+|Y_i^j|-1\equiv \defect_{T'}(b')+|Y_i^j|-1\equiv \defect_{\Inv(T';Y_i^j)}(b') \mod 2$, hence $\defect_{\Inv(T';Y_i^j)}(b_\mu)\leq \defect_{\Inv(T';Y_i^j)}(b')-2$. 
    This yields $\defect_{\Inv(T';Y_i^{j+1})}(b_\mu)\leq \defect_{\Inv(T';Y_i^j)}(b_\mu)+1\leq \defect_{\Inv(T';Y_i^j)}(b')-2+1=p-(j+1)$.
\end{proofclaim}

Let $T''=\Inv(T';(Y_1,\ldots,Y_t))$. Observe that by definition of the sets $Y_i$, we have $\defect_{T''}(b)=0$ for all $b \in\bigcup_{i=1}^t B_i$. 
As $|\bigcup_{i=1}^t B_i|=qt\leq\frac{2}{3}k$, we obtain by Proposition~\ref{extend} that there is a $k$-arc-strong tournament $T'''$ on $V(T)$ such that all the edges in $\UG(T)\langle \bigcup_{i=1}^t B_i\rangle \cup \delta_{\UG(T)}(\bigcup_{i=1}^t B_i)$ have the same orientation in $T''$ and $T'''$. 
Further, by Proposition~\ref{fuzzu}, there is a collection of $r=|V(T)\setminus \bigcup_{i=1}^t B_i|-1$ subsets $Z_1,\ldots,Z_r$ of $V(T)\setminus\bigcup_{i=1}^t B_i$ such that $\Inv(T''\langle V(T)\setminus\bigcup_{i=1}^t B_i \rangle; Z_1,\ldots,Z_r)=T'''\langle V(T)\setminus \bigcup_{i=1}^t B_i \rangle$. 
We obtain that $\Inv(T;X_1,\ldots,X_{q^*},\linebreak Y_1,\ldots,Y_t,Z_1,\ldots,Z_r)=\Inv(T';Y_1,\ldots,Y_t,Z_1,\ldots,Z_r)=\Inv(T'';Z_1,\ldots,Z_r)=T'''$. 
As $T'''$ is $k$-arc-strong, we have  $\sinv_k'(T)\leq q^* +t+r$.
Now $q^*= \lceil \log^2 k\rceil$, $t\leq \frac{2k}{3q}=\frac{2k}{3\lceil\frac{1}{4}\log k\rceil}$, and $r = |V(T)| - |\bigcup_{i=1}^t B_i| -1 \leq 2k -(\frac{2}{3}k -q) = \frac{4}{3}k + q =  \frac{4}{3}k + \lceil\frac{1}{4}\log k\rceil$ by Claim~\ref{bgross}.
Thus $\sinv_k'(T)\leq\frac{4}{3}k + \lceil \log^2 k\rceil + \lceil\frac{1}{4}\log k\rceil +  \frac{2k}{3\lceil\frac{1}{4}\log k\rceil} \leq \frac{4}{3}k + \epsilon k$.
\end{proof}
In the following lemmas, we show that the probability of each of $E_1,E_2$ and $E_3$ is small.
\begin{lemma}\label{prob1}
    $\Pr(E_1)<\frac{1}{3}$.
\end{lemma}
\begin{proof}
Let $b \in B$. We first bound the probability that $\vec{b}=\vec{0}$. Due to the independent and uniform choice of $(X_1,\ldots,X_{q^*})$, we have $\Pr[\vec{b}=\vec{0}]=2^{-q^*}\leq 2^{-2\log k}=\frac{1}{k^2}$.

Now suppose that $\vec{b}\neq \vec{0}$. Once $\vec{b}$ has been revealed, for every $v \in V(T)$, let $\Gamma_v$ be a random variable which is 1 if $T'$ contains the arc $bv$ and 0 otherwise and let $\Gamma=\sum_{v \in V(T-b)}\Gamma_v$. Observe that, due to the independent and uniform choice of $(X_1,\ldots,X_{q^*})$, we have $\Pr(\Gamma_v=1)=\frac{1}{2}$ for all $v \in V(T-b)$ and the $\Gamma_v$ are independent. This yields that $\Gamma\sim \Bin(2k,\frac{1}{2})$. Further, observe that $\defect_{T'}(b)\leq 2 \sqrt{k} \log k $ if and only if $k-2 \sqrt{k} \log k\leq \Gamma \leq k+2 \sqrt{k} \log k$. By Chernoff's Bound (Proposition~\ref{chernoff}), we obtain 

\begin{align*}
    \Pr\left[\Gamma<k-2 \sqrt{k} \log k\right]&\leq \Pr \left[\Gamma<\left(1-\frac{2 \sqrt{k} \log k}{k}\right)k\right]\\
    &\leq \exp\left(-\frac{4 k \log ^2 k}{2k^2}k\right)\\
    &\leq \exp(-2 \log k)\\
    &\leq \exp(-2 \ln k)\\
    & = \frac{1}{k^2}.
\end{align*} 
Similarly, since $2k-\Gamma \sim \Bin(2k,\frac{1}{2})$, we obtain that $\Pr[\Gamma>k+2 \sqrt{k} \log k]\leq \frac{1}{k^2}$.
This yields that 
\begin{align*}
    \Pr[\defect_{T'}(b)&>2 \sqrt{k} \log k]\\&=\Pr[\vec{b}=\vec{0}] \cdot \Pr[\defect_{T'}(b)>2 \sqrt{k} \log k\mid \vec{b}=\vec{0}]+\Pr[\vec{b}\neq 0] \cdot \Pr[\defect_{T'}(b)>2 \sqrt{k} \log k\mid \vec{b}\neq \vec{0}]\\
    &\leq \Pr[\vec{b}=\vec{0}]+\Pr[\Gamma<k-2 \sqrt{k} \log k\mid \vec{b}\neq \vec{0}]+\Pr[\Gamma>k+2 \sqrt{k} \log k \mid \vec{b}\neq \vec{0}]\\
    &\leq \frac{3}{k^2}.
\end{align*}

As there are $k+1$ vertices contained in $B$, by the Union Bound (Proposition~\ref{union}),
$\Pr(E_1)\leq (k+1)\frac{3}{k^2}< \frac{1}{3}$, as $k \geq 10$.
\end{proof}

\begin{lemma}\label{prob2}
    $\Pr(E_2)<\frac{1}{3}$.
\end{lemma}
\begin{proof}
Let $b_1,\ldots,b_q \in B$. Recall that $\vec{b}_1,\ldots,\vec{b}_q$ are independent and uniformly distributed. For $i\in [q]$, observe that if $\vec{b}_1,\ldots,\vec{b}_i$ are linearly independent, they span a vector space containing $2^{i}$ elements. Hence 

\begin{align*}
    \Pr[\vec{b}_1,\ldots,\vec{b}_q\text{ linearly dependent}]&=\sum_{i=1}^q \Pr[\vec{b}_1,\ldots,\vec{b}_i\text{ linearly dependent}\mid \vec{b}_1,\ldots,\vec{b}_{i-1}\text{ linearly independent }] \\
    & \qquad \qquad \times \Pr[\vec{b}_1,\ldots,\vec{b}_{i-1}\text{ linearly independent}]\\
    &\leq \sum_{i=1}^q \Pr[\vec{b}_1,\ldots,\vec{b}_i\text{ linearly dependent}\mid \vec{b}_1,\ldots,\vec{b}_{i-1}\text{ linearly independent}]\\
    &=\sum_{i=1}^q \frac{2^{i-1}}{2^{q^*}}\\
    &\leq \frac{2^q}{2^{q^*}}\\
    &\leq 2^{-2 \log (k) q} ~~~~\mbox{because $q^*\geq  \log^2 k \geq 
    3\log(k) q\geq 2 \log(k) q+q$}\\
    &= k^{-2q}.
\end{align*}

Observe that there are $\binom{k+1}{q}$ possibilities to choose $b_1,\ldots,b_q$. Hence the Union Bound (Proposition~\ref{union}) yields $\Pr(E_2)\leq \binom{k+1}{q} k^{-2q}\leq (2k)^qk^{-2q}=\left(\frac{2}{k}\right)^q< \frac{1}{3}$ 
as  $q\geq 1$ and $k \geq 7$.
\end{proof}

\begin{lemma}\label{prob3}
    $\Pr(E_3)<\frac{1}{3}$.
\end{lemma}

\begin{proof}
Let $b_1,\ldots,b_q \in B$ such that $\vec{b}_1,\ldots,\vec{b}_q$ are linearly independent and let $(\star_1,\ldots,\star_q)\in \{+,-\}^q$. Once $\vec{b}_1,\ldots,\vec{b}_q$ revealed, for every $a \in A$, let $\Gamma_a$ be a random variable that is 1 if $a \in \bigcap_{i=1}^q N_{T'}^{\star_i}(b_i)$ and 0 otherwise and let $\Gamma=\sum_{a \in A}\Gamma_a$. 
Observe that, due to the independent and uniform choice of $(X_1,\ldots,X_{q^*})$ and by Proposition~\ref{linalg}, we have that $\Pr[\Gamma_a=1]=\frac{1}{2^q}$ for all $a \in A$ and that the $\Gamma_a$ are independent. This yields that $\Gamma\sim \Bin(k,\frac{1}{2^q})$. Further observe that $\Gamma=|\bigcap_{i=1}^q N_{T'}^{\star_i}(b_i)\cap A|$.
Now, since $10 \log k \leq k^{\frac{1}{8}}$ and $q = \lceil\frac{1}{4}\log k\rceil \leq \frac{3}{8}\log k$, we obtain $5 \sqrt{k} \log k \leq \frac{1}{2}k^{\frac{5}{8}}=\frac{1}{2}k2^{-\frac{3}{8}\log k} \leq \frac{1}{2}k2^{-q}$.
By Chernoff's Bound (Proposition~\ref{chernoff}), we obtain 

\begin{align*}
\Pr[\Gamma<5 \sqrt{k} \log k]&\leq \Pr\left[\Gamma<\left(1-\frac{1}{2}\right)k2^{-q}\right]\\
&\leq \exp\left(-\frac{1}{8}k2^{-q}\right)\\
&\leq \exp\left(-\frac{1}{8}\sqrt{k}\right)\\
&\leq \exp(-2 q\log k) ~~~~~ \mbox{because $\sqrt{k} \geq 16 \log k \lceil \frac{1}{4}\log k\rceil$}\\
&\leq \exp(-2 q\ln k)\\
& = k^{-2q}.
\end{align*} 

Observe that there are $\binom{k+1}{q}$ possibilities to choose $b_1,\ldots,b_q$ and for each of those, there are $2^q$ possibilities to choose $(\star_1,\ldots,\star_q)$. Hence the Union Bound (Proposition~\ref{union}) yields $\Pr(E_3)\leq \binom{k+1}{q}2^qk^{-2q}\leq (2k)^q2^qk^{-2q}= \left(\frac{4}{k}\right)^q<\frac{1}{3}$, as $q \geq 1$ and $k \geq 13$.
\end{proof}

We are now ready to conclude the proof of Theorem~\ref{mkstrich}. By Lemmas~\ref{prob1},~\ref{prob2} and~\ref{prob3} and  the Union Bound (Proposition~\ref{union}), we obtain that, with positive probability, none of $E_1$, $E_2$ and $E_3$ occur. Hence the statement follows from Lemma~\ref{erfolg}.
\end{proof}

\subsubsection{Bigger tournaments}\label{big}

We now show how to extend this result to larger tournaments and derive Theorem~\ref{thm:upperM'} from Theorem~\ref{mkstrich}.

\borneMprime*

\begin{proof}
Let $\epsilon >0$. 
We shall prove that there is an integer $k^*$ such that for every $k \geq k^*$ and for every tournament $T$ on at least $2k+1$ vertices, we have $\sinv'_k(T)\leq (\frac{4}{3}+\epsilon)k$. 

By Theorem~\ref{mkstrich}, there is an integer $k_0$ such that for every $k \geq k_0$ and for every tournament $T$ on at least $2k+1$ vertices, we have $\sinv'_k(T)\leq (\frac{4}{3}+\frac{\epsilon}{2})k$. Further, let $k_1$ be an integer such that $\frac{\epsilon}{2}k \geq 6\log k\sqrt{k}+1$, $\sqrt{k}\geq 3 \log k$ and $\log k \geq 5$ hold for all $k \geq k_1$. 
Set $k^*= \max\{k_0,k_1\}$.

We now fix an integer $k \geq k^*$ and prove that for every tournament $T$ on $n \geq 2k+1$ vertices, we have $\sinv'_k(T)\leq (\frac{4}{3}+\frac{\epsilon}{2})k+\min\{n-(2k+1),6\log k\sqrt{k}+1\}$ from which the statement follows immediately. Suppose that $T$ is a tournament that does not satisfy this statement and whose number $n$ of vertices is minimum with respect to this property.
\begin{case}
$n > 4k-2$. 
\end{case}
By Proposition~\ref{4k-2}, 
there is some $v \in V(T)$ with $\min\{d_T^+(v),d_T^-(v)\}\geq k$. By the minimality of $T$, there is a collection $\mathcal{X}$ of $(\frac{4}{3}+\frac{\epsilon}{2})k+\min\{(n-1)-(2k+1),6\log k\sqrt{k}+1\}$ subsets of $V(T-v)$ such that $\Inv(T-v;\mathcal{X})$ is $k$-arc-strong. By Proposition~\ref{lem:kstrong+}, we obtain that $\Inv(T;\mathcal{X})$ is $k$-arc-strong. As $(\frac{4}{3}+\frac{\epsilon}{2})k+\min\{(n-1)-(2k+1),6\log k\sqrt{k}+1\} \leq (\frac{4}{3}+\frac{\epsilon}{2})k+\min\{n-(2k+1),6\log k\sqrt{k}+1\}$, we obtain a contradiction.
\begin{case}
$2k+1+6\log k \sqrt{k}< n \leq 4k-2$. 
\end{case}
Let $A$ be an arbitrary subset of $V(T)$ of size $2k+1+\lfloor6\sqrt{k} \log k\rfloor$ and let $B=V(T)\setminus A$. We now choose a subset $X$ of $A$ uniformly at random and let $T'=\Inv(T;B \cup X)$. Now consider some $b \in B$. For every $a \in A$, let $\Gamma_a$ be the probability that $ab \in A(T')$ and observe that $\Gamma=\sum_{a \in A}\Gamma_a$ is exactly $|N^-_{T'}(b)\cap A|$.
Due to the uniform choice of $X$, we have $\Pr[ab \in A(T)]=\frac{1}{2}$ and that the $\Gamma_a$ are independent. This yields $\Gamma \sim \Bin(2k+1+\lfloor6\sqrt{k} \log k\rfloor,\frac{1}{2})$. By Chernoff's Bound (Proposition~\ref{chernoff}), we obtain 

\begin{align*}
\Pr[\Gamma<k]&\leq \Pr\left[\Gamma<\left(1-\frac{ \sqrt{k} \log k}{k}\right)\frac{1}{2}\left(2k+1+\lfloor6\sqrt{k} \log k\rfloor\right)\right]\\
&\leq \exp(-\log^2(k)/2)\\
&\leq \exp(-2 \log k) ~~~~~~~ \mbox{because $k\geq 16$}\\
&\leq \exp(-2 \ln k)\\
& = \frac{1}{k^2}.
\end{align*} 
Similarly, we obtain $\Pr[|N^+_{T'}(b)\cap A|<k]\leq \frac{1}{k^2}$. As $B$ contains at most $(4k-3)-(2k+1+\lfloor6\sqrt{k} \log k\rfloor)\leq 2k$ elements, the Union Bound (Proposition~\ref{union}) yields that the probability that there is at least one $b \in B$ with $\min\{|N^-_{T'}(b)\cap A|,|N^+_{T'}(b)\cap A|\}\leq k$ is at most $\frac{2}{k^2}2k=\frac{4}{k}<1$, as $k\geq 5$.
Hence, with positive probability, we have $\min\{|N^-_{T'}(b)\cap A|,|N^+_{T'}(b)\cap A|\}\geq k$ for all $b \in B$. 
Thus there is $X_0$ such that every vertex of $B$ has at least $k$ in-neighbours and $k$ out-neighbours in $A$ in the tournament $T'_0= \Inv(T;B \cup X_0)$.
Further, by the minimality of $T$ and as $|V(T_0)|-2k+1=6 \sqrt{k}\log k$, there is a collection $\mathcal{X}$ of $\sinv'_k(T_0\langle A\rangle)\leq (\frac{4}{3}+\frac{\epsilon}{2})k+\lfloor6\log k\sqrt{k}\rfloor$ sets such that $\Inv(T'_0\langle A \rangle;\mathcal{X})$ is $k$-arc-strong. Hence, by Proposition~\ref{lem:kstrong+}, we obtain that $\Inv(T;\{X_0\} \cup \mathcal{X})$ is $k$-arc-strong, a contradiction.  


\begin{case}
$2k+2\leq n \leq 2k+6\sqrt{k} \log k$. 
\end{case}
Let $v \in V(T)$ be an arbitrary vertex. We can then find a set $X \subseteq V(T)$ such that for $T'=\Inv(T;X)$, we have $\min\{d_{T'}^+(v),d_{T'}^-(v)\}\geq k$. Further, by the minimality of $T$, there is a collection $\mathcal{X}$ of $\sinv'_k(T'-v)\leq (\frac{4}{3}+\frac{\epsilon}{2})k+n-(2k+1)-1$ sets such that $\Inv(T'-v; \mathcal{X})$ is $k$-arc-strong. Hence by Proposition~\ref{lem:kstrong+}, we obtain that $\Inv(T;\{X\}\cup \mathcal{X})$ is $k$-arc-strong, a contradiction.
\end{proof}

\section{Upper bounds on \texorpdfstring{$m_k(n)$}{mk(n)}\label{sec:upper_bound_Mkn}}
In this section, we prove several results showing that tournaments on significantly more than $2k$ vertices can be made $k$-strong by a small number of inversions. More precisely, in Subsection~\ref{first} we prove Theorem~\ref{thm:s1} and Proposition \ref{nk1unter}, in Subsection~\ref{nk10} we prove Theorem~\ref{nk1}, in Subsection~\ref{nk3sec} we prove Theorem~\ref{nk3}, and in Subsection~\ref{epsgross}, we prove Theorem~\ref{thm:2+eps}.

 \subsection{First upper bounds on \texorpdfstring{$m_k(n)$}{mk(n)}}\label{first}

In this subsection, we first establish that, for every fixed $k$, every tournament which is sufficiently large in comparison to $k$ can be made $k$-strong by a single inversion. While this result ( Theorem~\ref{thm:s1}) is clearly weaker than Theorem~\ref{nk1}, it justifies some of the notation used later on, and may serve as a warm-up exercise of the more involved proof of Theorem~\ref{nk1}.

\begin{theorem}\label{thm:s1}
Let $n,k$ be positive integers and let $T$ be a tournament on $n$ vertices.
If $n \geq (2k-1)2^{2k}$, then $\sinv'_k(T) \leq \sinv_k(T) \leq 1$.
\end{theorem}
\begin{proof}
Let $T$ be a tournament of order $n\geq (2k-1)2^{2k}$.

  It is easy and well-known that if $D$ is an acyclic digraph, $x$ a source in $D$, and $D-x$ is contained (as a subdigraph) in every tournament of order $n$, then $D$ is contained (as a subdigraph) in every tournament of order $2n$.  
An easy induction yields that $T$ contains three sets
  $A_1, A_2, A_3$ such that $A_1\Ra (A_2\cup A_3)$ and
  $A_2\Ra A_3$ with $|A_1|=|A_3| = k$ and $|A_2|=2k-1$.
Set $A=A_1\cup A_2 \cup A_3$.
Let $I$ be the set of vertices in $V(T)\setminus A$ that have either fewer than $k$ out-neighbours in $A$ or fewer than $k$ in-neighbours in $A$.
Let $X=A_1\cup A_3\cup I$.
Let us prove that $T'=\Inv(T; X)$ is $k$-strong.

In $T'$, we have $A_1 \Ra A_2 \Ra A_3\Ra A_1$. Since the three sets $A_1$, $A_2$, and $A_3$ have size at least $k$, the tournament $T'\langle A\rangle$ is $k$-strong.
Now consider a vertex $v$ in $V(T)\setminus A$.
If $v\notin I$, then no arcs incident to $v$ have been reversed so its in- and out-degree have been unchanged and are at least $k$ by definition of $I$.
If $v\in I$, then all the arcs between $v$ and $A_1\cup A_3$ have been reversed and those between $v$ and $A_2$ are unchanged.
If $v$ has fewer than $k$ out-neighbours in $A$ in $T$, then $|N^+_{T'}(v)\cap A| \geq |N^-_T(v) \cap (A_1\cup A_3)| \geq 2k - d^+_T(v) \geq k$, and $|N^-_{T'}(v)\cap A|\geq |N^-_T(v) \cap A_2| \geq 2k-1 - d^+_T(v) \geq k$.
Similarly, if $v$ has fewer than $k$ in-neighbours in $A$ in $T$, then 
$|N^+_{T'}(v)\cap A|\geq k$ and $|N^-_{T'}(v)\cap A|\geq k$.
Thus, by 
Lemma~\ref{lem:kstrong+}, $T'$ is $k$-strong.
Hence $\sinv_k(T) \leq 1$.
\end{proof}

\subsection{Linear upper and lower bounds on \texorpdfstring{$N_k(1)$}{Nk(1)}}\label{nk10}

In this subsection, we shall prove  Theorem~\ref{nk1} which states that $N_k(1)\leq 19k -2$. To prove it we need some preliminaries.

For a digraph $D$, let $\sigma=(v_1,v_2, \ldots, v_n)$ be an ordering of the vertices of $D$. An arc $v_iv_j$ is {\bf forward} (according to $\sigma$) if $i<j$ and {\bf backward} (according to $\sigma$) if $j<i$.
A {\bf median order} of $D$ is an ordering of the vertices of $D$ with the maximum number of forward arcs, or, equivalently, the minimum number of backward arcs.

Let us note basic well-known properties of median orders of
tournaments (the ``feedback property'' in \cite{HaTh00}).

\begin{lemma}\label{lem:median}
 Let $T$ be a tournament and $(v_1,v_2, \ldots, v_n)$ a
median order of $T$. Then, for any two indices $i,j$ with $1 \leq i <
j \leq n$:
\medskip
\begin{enumerate}
\item[\rm (M1)] $(v_i,v_{i+1},\ldots,v_j)$ is a median order of the
  induced subtournament $T\langle \{v_i,v_{i+1},\ldots,v_j\}\rangle$.\\
  
\item[\rm (M2)] the vertex $v_i$ is an in-neighbour of at least half of the vertices
  $v_{i+1},v_{i+2},\ldots,v_j$, and the vertex $v_j$ is  an out-neighbours of at least half of the vertices $v_i,v_{i+1},\ldots,v_{j-1}$.  In particular, for every $1 \leq i \leq n-1$, $v_i v_{i+1} \in A(T)$. 

\end{enumerate}
\end{lemma}

Given two vertices $u,v$ of some digraph $D$, we say that $v$ is {\bf reachable} from $u$ if $D$ contains a directed $uv$-path. For $v \in V(D)$, we use $R^+_D(v)$ (resp. $R^-_D(v)$) for the set of vertices which are reachable from $v$ in $D$ (resp. from which $v$ is reachable in $D$). 
Note that $v\in R^+_D(v) \cap R^-_D(v)$.

\begin{lemma}\label{lem:n-2F}
Let $T$ be a tournament with median order $(v_1, v_2, \ldots , v_n)$.
Let $F$ be a subset of vertices such that $v_1\notin F$. 
Then $|R^+_{T-F}(v_1)| \geq n - 2 |F|$.
\end{lemma}
\begin{proof}
We prove the result by induction on $n+|F|$, the result holding trivially by (M2) if $|F|=0$.

If all the out-neighbours of $v_1$ in $T$ are in $|F|$, then by (M2), $|N_T^-(v_1)| \leq |N_T^+(v_1)|  \leq |F|$. Hence $n-1 \leq 2|F|$, and the result holds.
Henceforth we may assume that $v_1$ has an out-neighbour not in $F$.
Let $i_0$ be the smallest index of such a vertex.
Let $T_0= T\langle \{v_1, \dots , v_{i_0-1}\}\rangle$, $T_1= T\langle \{v_{i_0}, \dots , v_{n}\}\rangle$,  $F_0=F\cap V(T_0)$ and $F_1=F\cap V(T_1)$.
By (M1), $(v_1, v_2, \ldots , v_{i_0-1})$ is a median order of $T_0$ and $(v_{i_0},  \ldots , v_n)$ is a median order of $T_1$.
By definition of $i_0$, all out-neighbours of $v_1$ in $T_0$ are in $F_0$. Thus, as above, we have $i_0-2 \leq 2|F_0|$.
By the induction hypothesis, $|R^+_{T_1-F_1}(v_{i_0})| \geq n - i_0 + 1 - 2 |F_1|$.
Now $R^+_{T_1-F_1}(v_{i_0}) \cup \{v_1\} \subseteq R^+_{T-F}(v_1)$. Hence
 $$|R^+_{T-F}(v_1)| \geq |R^+_{T_1-F_1}(v_{i_0})| +1 \geq n - i_0 +2 - 2|F_1| \geq n - 2|F_0| - 2|F_1| = n - 2|F|.$$
\end{proof}




We are now ready to prove the main lemma of the proof of Theorem~\ref{nk1}.


\begin{lemma}\label{6ab}
    Let $k$ be a positive integer, let $T$ be a tournament on $12k$ vertices and let $(A,B)$ be a bipartition of $V(T)$ such that $|A|=|B|=6k$. 
    Then there is a set $X \subseteq V(T)$ with $|X \cap A|=|X \cap B|=2k$ such that for $T'=\Inv(T;X)$ and for any $Y \subseteq V(T)$ with $|Y|\leq k-1$, we have that $T'-Y$ contains a directed path from $a$ to $B\setminus Y$ for every $a \in A\setminus Y$,
    and $T'-Y$ contains a directed path from $A\setminus Y$ to $b$ for every $b \in B\setminus Y$.
\end{lemma}
\begin{proof}
    Let $(a_1,\ldots,a_{6k})$ be a median order of $T\langle A \rangle$ and let $(b_1,\ldots,b_{6k})$ be a median order of $T\langle B \rangle$. 
    Let $A_0$ be the set of vertices in $\{a_{4k+1},\ldots,a_{6k}\}$ which have fewer than $k$ out-neighbours in $B$ in $T$. Further, let $A_1=\{a_1,\ldots,a_{2k-|A_0|}\}$. 
    Observe that $|A_0 \cup A_1|=|A_0|+|A_1|=2k$. 
    Similarly, let $B_0$ be the set of vertices in $\{b_1,\ldots,b_{2k}\}$ that have fewer than $k$ in-neighbours in $A$ and let $B_1=\{b_{4k+|B_0|+1},\ldots,b_{6k}\}$. 
    Observe that $|B_0 \cup B_1|=|B_0|+|B_1|=2k$. Let $X=A_0 \cup A_1 \cup B_0 \cup B_1$ and let $T'=\Inv(T;X)$. 
    Consider any $Y \subseteq V(T)$ with $|Y|\leq k-1$. 
    In order show that $T'$ has the desired properties, by symmetry, it suffices to prove that $T'-Y$ contains a directed path from $a$ to $B\setminus Y$ for every $a \in A\setminus Y$. 
    Suppose by contradiction that this is not true. Then there is a largest integer $i \in [6k]$ such that $a_i \in A\setminus Y$ and $T'-Y$ does not contain a directed path from $a_i$ to $B\setminus Y$. We will distinguish several cases.

\setcounter{case}{0}
    \begin{case}
        $i \in \{4k+1,\ldots,6k\}$ and $a_i \in A\setminus A_0 $.
    \end{case}
    In this case, by the choice of $A_0$, we have 
    \begin{align*}
        |(N_{T'}^+(a_i)\cap B)\setminus Y|&\geq |N_{T'}^+(a_i)\cap B|-|Y|\\
        &= |N_{T}^+(a_i)\cap B|-|Y|\\
        &\geq k-(k-1)\\
        &=1,
    \end{align*}
    so $a_i$ has an out-neighbour in $B\setminus Y$ in $T'-Y$, a contradiction.
    \begin{case}
        $i \in \{4k+1,\ldots,6k\}$ and $a_i \in A_0$.
    \end{case}
    In this case, by the choice of $A_0$, we have
    \begin{align*}
        |(N_{T'}^+(a_i)\cap B)\setminus Y|&\geq |N_{T'}^+(a_i)\cap B|-|Y|\\
        &\geq|N_{T'}^+(a_i)\cap (B_0 \cup B_1)|-|Y|\\
        &=|B_0 \cup B_1|-|N_{T'}^-(a_i)\cap (B_0 \cup B_1)|-|Y|\\
        &=|B_0 \cup B_1|-|N_{T}^+(a_i)\cap (B_0 \cup B_1)|-|Y|\\
        &\geq|B_0 \cup B_1|-|N_{T}^+(a_i)\cap B|-|Y|\\
        &\geq 2k-(k-1)-(k-1)\\
        &=2,
    \end{align*}
    so $a_i$ has an out-neighbour in $B\setminus Y$ in $T'-Y$, a contradiction.
    
    \begin{case}
        $i \in \{2k-|A_0|+1,\ldots,4k\}$.
    \end{case}
    As $(a_1,\ldots,a_{6k})$ is a median order of $T\langle A \rangle$, and by (M2) applied to $T\langle A \rangle$, we have 
     \begin{align*}
    |(N_{T'}^+(a_i)\cap \{a_{i+1},\ldots,a_{6k}\})\setminus Y|
    &\geq |N_{T'}^+(a_i)\cap \{a_{i+1},\ldots,a_{6k}\}|-|Y|\\
    &= |N_{T}^+(a_i)\cap \{a_{i+1},\ldots,a_{6k}\}|-|Y|\\
    &\geq \frac{1}{2} |\{a_{i+1},\ldots,a_{6k}\}|-|Y|\\
    & \geq \frac{1}{2} (6k-i)-(k-1)\\
    &=2k+1-\frac{i}{2}\\
    &\geq 2k+1-2k\\
    &=1.
    \end{align*}
    Hence there is some $j>i$ such that $a_j\in A\setminus Y$ and $T'-Y$ contains the arc $a_ia_j$. By the maximality of $i$, there is a directed path from $a_j$ to $B\setminus Y$ in $T'-Y$. Hence $T'-Y$ also contains a directed path from $a_i$ to $B\setminus Y$, a contradiction.
    \begin{case}
    $i \in \{1,\ldots,2k-|A_0|\}$.
    \end{case}
     As $(a_1,\ldots,a_{6k})$ is a median order of $T\langle A \rangle$ and by (M2) applied to $T\langle A \rangle$, we have 
     \begin{align*}
    |(N_{T'}^+(a_i)\cap \{a_{i+1},\ldots,a_{6k}\}) \setminus Y| & \geq |N_{T'}^+(a_i)\cap \{a_{i+1},\ldots,a_{6k}\}|-|Y|\\
    &\geq |N_{T}^+(a_i)\cap \{a_{i+1},\ldots,a_{6k}\}|-|(A_0 \cup A_1)\cap \{a_{i+1},\ldots,a_{6k}\}|-|Y|\\
    &\geq \frac{1}{2} |\{a_{i+1},\ldots,a_{6k}\}|-|(A_0 \cup A_1)\cap \{a_{i+1},\ldots,a_{6k}\}|-|Y|\\
    & \geq \frac{1}{2} (6k-i)-(2k-i)-(k-1)\\
    &=\frac{i}{2}+1\\
    &\geq 1.
    \end{align*}
    Hence there is some $j>i$ such that $a_j\in A\setminus Y$ and $T'-Y$ contains the arc $a_ia_j$. By the maximality of $i$, there is a directed path from $a_j$ to $B\setminus Y$ in $T'-Y$. Hence $T'-Y$ also contains a directed path from $a_i$ to $B\setminus Y$, a contradiction.
\end{proof}

    We are now ready to prove Theorem~\ref{nk1}, which states $N_k(1)\leq 19k-2$.
\begin{proof}[Proof of Theorem~\ref{nk1}]
    Assume $n\geq 19k-2$.
    Let $(v_1, \ldots , v_n)$ be a median order of $T$.
    Let $B=\{v_1, \dots , v_{6k}\}$, $A=\{v_{n-6k+1}, \ldots , v_{n}\}$ and $C=V(T)\setminus (A\cup B)$. We have $|C| \geq 7k-2$.
    By Lemma~\ref{6ab}, applied to $T\langle A\cup B\rangle$, there is a subset $X$ of $A\cup B$ such that,
    for $T_1 = \Inv(T;X)$, we have
    \begin{itemize}
        \item[(i)] $|X \cap B|=|X \cap A|=2k$ ; 
        \item[(ii)] for any $Y \subseteq V(T)$ with $|Y|\leq k-1$, $T_1-Y$ contains a directed path from $a$ to $B\setminus Y$ for every $a \in A\setminus Y$; 
        \item[(iii)] for any $Y \subseteq V(T)$ with $|Y|\leq k-1$, $T_1-Y$ contains a directed path from $A\setminus Y$ to $b$ for every $b \in B\setminus Y$.
    \end{itemize}
    Let us now prove that $T_1$ is $k$-strong, which implies $\sinv_k(T) \leq 1$.
    Let $F$ be a set of at most $k-1$ vertices of $T_1$.
    We first need the following intermediate result.
    \begin{claim}\label{nett}
        Let $b \in B\setminus F$ and $a \in A\setminus F$. Then $T_1-F$ contains a directed path from $b$ to $a$.
    \end{claim}

    \begin{proofclaim}
    By construction, there is some $i \in \{1,\ldots,6k\}$ such that $b=v_i$. Let $S=((X \cap B)\cup F)\cap \{v_{i+1},\ldots,v_{n-6k}\}$ and let $T'=T\langle\{v_{i},\ldots,v_{n-6k}\}\rangle$. By (M2), every vertex in $S$ and every vertex in $C\setminus F$ is reachable from $b$ in $T'$. On the other hand, there is obviously no vertex in $S$ reachable from $b$ in $T'-S$. 
    Further, by (M1), $(v_i,\ldots,v_{n-6k})$ is a median order of $T'$.
    By Lemma~\ref{lem:n-2F} applied to $T'$ and $S$, this yields  
    \begin{align*}
        2|S|&\geq |R^+_{T'}(b)|-|R^+_{T'-S}(b)|\\
        &= |R^+_{T'}(b)\cap B|-|R^+_{T'-S}(b)\cap B|\\
        & \qquad +|R^+_{T'}(b)\cap (C \cap F)|-|R^+_{T'-S}(b)\cap (C \cap F)|\\
        & \qquad +|R^+_{T'}(b)\cap (C \setminus F)|-|R^+_{T'-S}(b)\cap (C \setminus F)|\\
        &\geq  |S| +|C \setminus F|-|R^+_{T'-S}(b)\cap (C \setminus F)|.
    \end{align*}
    For every $v \in C \setminus F$, if $v$ is reachable from $b$ in $T'-S$, then $v$ is clearly reachable from $b$ in $T_1-F$. Since $|X \cap B|\leq 2k$ and $|F|\leq k-1$, we obtain 
    \begin{align*}
        |R^+_{T_1-F}(b)\cap (C \setminus F)|&\geq  |R^+_{T'-S}(b)\cap (C \setminus F)|\\
        &\geq |C \setminus F|-|S|\\
        &\geq |C \setminus F|-(|X \cap B|+|F|)\\
        &\geq |C \setminus F|-\big(2k+(k-1)\big)\\
        &= |C \setminus F|-(3k-1).
    \end{align*}
    A similar argument shows that $ |R^-_{T_1-F}(a)\cap (C \setminus F)|\geq |C \setminus F|-(3k-1)$. As $|C|\geq n-12k\geq 7k-2$, we obtain
     \begin{align*}
        |(R^+_{T_1-F}(b)\cap (C \setminus F)) \cap (R^-_{T_1-F}(a)\cap (C \setminus F))|&
        =|R^+_{T_1-F}(b)\cap (C \setminus F)|+|R^-_{T_1-F}(a)\cap (C \setminus F)|\\
        &\qquad-|(R^+_{T_1-F}(b)\cap (C \setminus F)) \cup (R^-_{T_1-F}(a)\cap (C \setminus F))|\\
        & \geq |R^+_{T_1-F}(b)\cap (C \setminus F)|+|R^-_{T_1-F}(a)\cap (C \setminus F)|-|C \setminus F|\\
        &\geq 2\big(|C \setminus F|-(3k-1)\big)-|C \setminus F|\\
        &=|C \setminus F|-2(3k-1)\\
        &\geq |C|-|F|-(6k-2)\\
        & \geq (7k-2)-(k-1)-(6k-2)\\
        & = 1.
    \end{align*}
    Hence there is a vertex $v^* \in \{v_{6k+1},\ldots,v_{n-6k}\}\cap (R^+_{T_1-F}(b)\cap (C \setminus F)) \cap (R^-_{T_1-F}(a)\cap (C \setminus F))$. By definition, $T_1-F$ contains a directed path from $b$ to $v^*$ and a directed path from $v^*$ to $a$. Hence $T_1-F$ contains a directed path from $b$ to $a$.
    \end{proofclaim}
      
    We are now ready to show that $T_1-F$ is strong.
    Let $x$ and $y$ be two vertices in $T_1-F$. It suffices to show that $y$ is reachable from $x$ in $T_1-F$.
    
    We first show that there is a path from $x$ to $B \setminus F$. Clearly, we may suppose that $x \in (A \cup C)\setminus F.$ 
    Then, since $|A| > 2|F|$, Lemma~\ref{lem:n-2F} implies that
    there is a vertex $x' \in A\setminus F$ reachable from $x$ in $T-F-B$, and so in $T_1-F$.
    By (ii), there is a directed path from $x'$ to a vertex $u$ in $B\setminus F$ in $T_1-F$.
    Hence there is a directed path $P_x$ from $x$ to $u$ in $T_1-F$.
    Similarly, by directional duality, in $T_1-F$, there is a directed path $P_y$ from a vertex $w\in A\setminus F$ to $y$.
    
    Finally, by Claim~\ref{nett}, there exists a path $Q$ from $u$ to $w$ in $T_1 - F$. Then
    $P_x Q P_y$ is a path from $x$ to $y$ in $T_1-F$.
    This proves that $T_1-F$ is strong.
\end{proof}

Finally, we give the proof of Proposition~\ref{nk1unter}, which states that $N_k(1)\geq 5k-2$ and $N_k'(1)\geq 4k-1$.
\begin{proof}[Proof of Proposition~\ref{nk1unter}]
    Let $T_1$ be a tournament of order $5k-3$ whose vertex set has a partition $(A,B,C)$ such that $T_1\langle A \rangle$ and $T_1\langle C\rangle$ are eulerian tournaments of order $2k-1$, and $A\Ra B\cup C$ and $B\Ra C$.
We shall prove that $\sinv_k(T_1) >1$.

Assume for a contradiction that there is a set $X$ of vertices such that $T_1'=\Inv(T_1;X)$ is $k$-strong.
Every vertex of $A$ (resp. $C$) has in-degree (resp. out-degree) $k-1$ in $T_1$, and so belongs to $X$.
Thus $A\cup C\subseteq X$, and so $C\Ra A$ in $T_1'$. Hence $T_1'-B$ is not strong. Since $|B|=k-1$, $T_1'$ is not $k$-strong, a contradiction.

\medskip

Now let $T_2$ be a tournament of order $4k-2$ whose vertex set has a partition $(A,B)$ such that $T_2\langle A \rangle$ and $T_2\langle B\rangle$ are eulerian tournaments of order $2k-1$, and $A\Ra B$.
We shall prove that $\sinv_k'(T_1) >1$.

Assume for a contradiction that there is a set $X$ of vertices such that $T_2'=\Inv(T_2;X)$ is $k$-arc-strong.
Every vertex of $A$ (resp. $B$) has in-degree (resp. out-degree) $k-1$ in $T_2$, and so belongs to $X$.
Thus $X=V(T_2)$, and so $T_2'$ is isomorphic to $T_2$. Hence $T_2'$ is not $k$-arc-strong, a contradiction.
\end{proof}



\subsection{Better upper bound on \texorpdfstring{$N_k(3)$}{Nk(3)}}\label{nk3sec}
This section is dedicated to the proof of Theorem~\ref{nk3}. The structure is analogous to the proof of Theorem~\ref{nk1} in Section~\ref{nk10}.
First we show a result which is very similar to Lemma~\ref{6ab}.

\begin{lemma}\label{lem:4ab}
    Let $k$ be a positive integer, $T$ a tournament on $8k$ vertices and $(A,B)$ a bipartition of $V(T)$ such that $|A|=|B|=4k$. 
    There is a family ${\cal X}$ of three subsets of $V(T)$ such that the following hold with $T'= \Inv(T ; {\cal X})$. 
    \begin{itemize}
        \item[(i)] $T'\langle A \rangle = T\langle A\rangle$ and $T'\langle B \rangle = T\langle B\rangle$ ; 
        \item[(ii)] for any $Y \subseteq V(T)$ with $|Y|\leq k-1$, $T'-Y$ contains a directed path from $a$ to $B\setminus Y$ for every $a \in A\setminus Y$; 
        \item[(iii)] for any $Y \subseteq V(T)$ with $|Y|\leq k-1$, $T'-Y$ contains a directed path from $A\setminus Y$ to $b$ for every $b \in B\setminus Y$.
    \end{itemize}
\end{lemma}
\begin{proof}
    Let $(a_1,\ldots,a_{4k})$ be a median order of $T\langle A \rangle$ and let $(b_1,\ldots,b_{4k})$ be a median order of $T\langle B \rangle$. Let $A_0$ be the set of vertices in $\{a_{2k+1},\ldots,a_{4k}\}$ which have fewer than $k$ out-neighbours in $B$ in $T$. Further, let $A_1=\{a_1,\ldots,a_{2k-|A_0|}\}$. Observe that $|A_0 \cup A_1|=|A_0|+|A_1|=2k$. Similarly, let $B_0$ be the set of vertices in $\{b_1,\ldots,b_{2k}\}$ that have fewer than $k$ in-neighbours in $A$ and let $B_1=\{b_{2k+|B_0|+1},\ldots,b_{4k}\}$. Observe that $|B_0 \cup B_1|=|B_0|+|B_1|=2k$. Now let $X_1=A_0 \cup A_1$, $X_2= B_0 \cup B_1$,
    $X_3=X_1\cup X_2$, and ${\cal X} = (X_i)_{i\in [3]}$ Let $T'=\Inv(T;{\cal X})$. Observe that $T'$ is obtained from $T$ by reversing the arcs between $X_1$ and $X_2$. In particular, (i) holds. 

    We only show (ii) as (iii) follows symmetrically.
    Let $Y \subseteq V(T)$ with $|Y|\leq k-1$.  Suppose for the sake of a contradiction that (ii) does not hold. There is a largest integer $i \in [4k]$ such that $a_i \in A\setminus Y$ and $T'-Y$ does not contain a directed path from $a_i$ to $B\setminus Y$. We will distinguish several cases.

\setcounter{case}{0}
    \begin{case}
        $i \in \{2k+1,\ldots,4k\}$ and $a_i \in A\setminus A_0$.
    \end{case}
    In this case, by the choice of $A_0$, we have 
    \begin{align*}
        |(N_{T'}^+(a_i)\cap B)\setminus Y|&\geq |N_{T'}^+(a_i)\cap B|-|Y|\\
        &= |N_{T}^+(a_i)\cap B|-|Y|\\
        &\geq k-(k-1)\\
        &=1,
    \end{align*}
    so $a_i$ has an out-neighbour in $B\setminus Y$ in $T'-Y$, a contradiction.
    \begin{case}
        $i \in \{2k+1,\ldots,4k\}$ and $a_i \in A_0$.
    \end{case}
    In this case, by the choice of $A_0$, we have
    \begin{align*}
        |(N_{T'}^+(a_i)\cap B)\setminus Y|&\geq |N_{T'}^+(a_i)\cap B|-|Y|\\
        &\geq|N_{T'}^+(a_i)\cap (B_0 \cup B_1)|-|Y|\\
        &=|B_0 \cup B_1|-|N_{T'}^-(a_i)\cap (B_0 \cup B_1)|-|Y|\\
        &=|B_0 \cup B_1|-|N_{T}^+(a_i)\cap (B_0 \cup B_1)|-|Y|\\
        &\geq|B_0 \cup B_1|-|N_{T'}^+(a_i)\cap B|-|Y|\\
        &\geq 2k-(k-1)-(k-1)\\
        &=2,
    \end{align*}
    so $a_i$ has an out-neighbour in $B\setminus Y$ in $T'-Y$, a contradiction.
    
    \begin{case}
        $i \in [2k]$.
    \end{case}
    As $(a_1,\ldots,a_{4k})$ is a median order of $T\langle A \rangle=T'\langle A\rangle$, by (M2) we have 
     \begin{align*}
    |(N_{T'}^+(a_i)\cap \{a_{i+1},\ldots,a_{4k}\})\setminus Y|&\geq |N_{T'}^+(a_i)\cap \{a_{i+1},\ldots,a_{4k}\}|-|Y|\\
    & = |N_{T}^+(a_i)\cap \{a_{i+1},\ldots,a_{4k}\}|-|Y|\\
    & \geq \frac{1}{2} |\{a_{i+1},\ldots,a_{4k}\}|-|Y|\\
    & \geq \frac{1}{2} (4k-i)-(k-1)\\
    & = k+1-\frac{i}{2}\\
    & \geq k+1-k\\
    & = 1.
    \end{align*}
    Hence there is some $j>i$ such that $a_j\in A\setminus Y$ and $T'-Y$ contains the arc $a_ia_j$. By the maximality of $i$, there is a directed path from $a_j$ to $B\setminus Y$ in $T'-Y$. Hence $T'-Y$ also contains a directed path from $a_i$ to $B\setminus Y$, a contradiction.
\end{proof}

We now prove the upper bound on $N_k(3)$, which we restate here.

\nkthree*

\begin{proof}
    Let $T$ be a tournament on  $n\geq 11 k-2$ vertices.
    Let $(v_1, \ldots , v_n)$ be a median order of $T$.
    Let $B=\{v_1, \dots , v_{4k}\}$, $A=\{v_{n-4k+1}, \ldots , v_{n}\}$ and $C=V(T)\setminus (A\cup B)$. We have $|C| \geq 3k-2$.
    By Lemma~\ref{lem:4ab} applied to $T\langle A\cup B\rangle$, there is a family  ${\cal X}$ of three subsets of $A\cup B$ such that
    for $T_1 = \Inv(T;\mathcal{X})$ we have
    \begin{itemize}
        \item[(i)] $T_1 \langle A \rangle = T\langle A\rangle$ and $T_1\langle B \rangle = T\langle B\rangle$ ; 
        \item[(ii)] for any $Y \subseteq V(T)$ with $|Y|\leq k-1$, $T_1 -Y$ contains a directed path from $a$ to $B\setminus Y$ for every $a \in A\setminus Y$; 
        \item[(iii)] for any $Y \subseteq V(T)$ with $|Y|\leq k-1$, $T_1 -Y$ contains a directed path from $A\setminus Y$ to $b$ for every $b \in B\setminus Y$.
    \end{itemize}
    
    Let $T_1=\Inv(T; \mathcal{X})$.
    Let us now prove that $T_1$ is $k$-strong, which implies $\sinv_k(T) \leq 3$.
    Note that $T\langle A \cup C \rangle$, as well as $T\langle B \cup C \rangle$ are unchanged by the inversions.
    Let $F$ be a set of at most $k-1$ vertices of $T_1$. Let us show that $T_1-F$ is strong.
    Let $x$ and $y$ be two vertices in $T_1-F$. It suffices to show that $y$ is reachable from $x$ in $T_1-F$.
    
    Let us first show that there is a vertex $u \in B\setminus F$ that is reachable from $x$ in $T_1-F$. It is trivial if $x\in B$, so we may suppose that $x \in A\cup C$.
    Lemma~\ref{lem:n-2F} asserts that from any vertex of $(A \cup C)\setminus F$ one can reach a vertex of $A\setminus F$ in $T_1-F$. Thus there exists a vertex $x' \in A\setminus F$ reachable from $x$ in $T_1-F$. By (ii), there is a directed path from $x'$ to a vertex $u$ in $B\setminus F$ in $T_1 - F$.
    Hence there is a directed path $P_x$ from $x$ to $u$ in $T_1-F$.
    Similarly, by directional duality, in $T_1-F$, there is a directed path $P_y$ from a vertex $w\in A\setminus F$ to $y$.
    
    
    By Lemma~\ref{lem:n-2F}, we have $|R^+_{T\langle B \cup C\rangle - F}(u) \cap (C\setminus F)| \geq |C|-2|F\cap (B \cup C)| = |C\setminus F| -|F\cap C|-2|F \cap B|.$ Similarly, we obtain $|R^-_{T\langle A \cup C\rangle - F}(w) \cap (C\setminus F)|  \geq |C\setminus F| -|F\cap C|-2|F \cap A|.$ 
    
    This yields 
    
    \begin{align*}
    |R^+_{T\langle B \cup C\rangle - F}(u)\cap R^-_{T\langle A \cup C\rangle - F}(w)|&=|(R^+_{T\langle B \cup C\rangle - F}(u)\cap(C \setminus F))\cap (R^-_{T\langle A \cup C\rangle - F}(w)\cap(C \setminus F))|\\
    &=|R^+_{T\langle B \cup C\rangle - F}(u)\cap(C \setminus F)|+| R^-_{T\langle A \cup C\rangle - F}(w)\cap(C \setminus F)|\\ & \qquad-|(R^+_{T\langle B \cup C\rangle - F}(u)\cap(C \setminus F))\cap (R^-_{T\langle A \cup C\rangle - F}(w)\cap(C \setminus F))|\\
    &\geq (|C\setminus F| -|F\cap C|-2|F \cap B|)+(|C\setminus F| -|F\cap C|-2|F \cap A|)-|C\setminus F|\\
    &=|C\setminus F|-2|F|\\
    &\geq|C|-3|F|\\
    &\geq (3k-2)-3(k-1)\\
    &=1.
    \end{align*}
    Thus there is a vertex of $C\setminus F$ in $R^+_{T\langle B \cup C\rangle - F}(u)\cap R^-_{T\langle A \cup C\rangle - F}(w)$, and so there exists a directed path $Q$ from $u$ to $w$ in $T_1 - F$. 
    Then $P_x Q P_y$ is a directed path from $x$ to $y$ in $T_1-F$.
\end{proof}

All the results of the previous subsections imply the following.
\begin{corollary}\label{cor:m_k-upper}
    $m_k(n) \leq
    \left\{ 
    \begin{array}{ll}
        2k &\text{if~ $ n\geq 2k+1$,}\\
        3  &\text{if~ $ n \geq 11k - 2$,}\\
        1  &\text{if~  $ n\geq 19k - 2$.}
    \end{array}\right.$
\end{corollary}


\subsection{Upper bounds for \texorpdfstring{$k$}{k} large}\label{epsgross}
In this section, we show that if a tournament has at least $2k+1+\epsilon k$ vertices for some positive integer $k$ and some $\epsilon>0$, then it can be made $k$-strong by inverting a family of sets whose cardinality only depends on $\epsilon$.
The proof consists in drawing this family uniformly at random, under the constraint that every vertex is contained in at least one of the sets. 

To analyse this procedure we will need Chernoff's Bound (Proposition~\ref{chernoff}) as well as the two following technical lemmas.

\begin{lemma}\label{lemma:proba_2}
    Let $\vec{u} \neq \vec{v} \in \mathbb{F}_2^t \setminus \{\vec{0}\}$ and $x,y \in \mathbb{F}_2$ be fixed, and let $\vec{w} \in \mathbb{F}_2^t \setminus \{\vec{0}\}$ 
    be drawn uniformly at random. Then $\Pr[\vec{u} \cdot \vec{w} = x, \vec{v} \cdot \vec{w}=y] \geq \frac{1}{4}-\frac{3}{4}\frac{1}{2^t-1}$.
\end{lemma}

\begin{proof}
    As $\vec{u}\neq \vec{v}$, the mapping $\mathbb{F}_2^t \to \mathbb{F}_2^2, \vec{w} \mapsto (\vec{u} \cdot \vec{w}, \vec{v} \cdot \vec{w})$ is surjective and linear.
    As a consequence, there are $\frac{1}{4}2^t$ vectors $\vec{w} \in \mathbb{F}_2^t$ which satisfy $\vec{u} \cdot \vec{w}=x$ and $\vec{v} \cdot \vec{w} = y$.
    Thus by possibly removing the solution $\vec{w}=0$, we obtain $\Pr[\vec{u} \cdot \vec{w} = x, \vec{v} \cdot \vec{w}=y] \geq \frac{2^{t-2}-1}{2^t-1} =
    \frac{1}{4}-\frac{3}{4}\frac{1}{2^t-1}$.
\end{proof}

\begin{lemma}\label{lemma:proba_3}
    Let $\epsilon >0$, let $t \geq 16$ be an integer, and let $k \geq \frac{8t}{\epsilon}$ be an integer.
    Let $U,V \in (\mathbb{F}_2^{t} \setminus\{\vec{0}\})^{\lceil\epsilon k/8\rceil}$ be drawn uniformly at random and 
    $W \in \mathbb{F}_2^{\lceil\epsilon k/8\rceil \times \lceil\epsilon k/8\rceil}$ be fixed.
    Then $\Pr[U^\top \cdot V = W] \leq 2^{-\epsilon t k/128}$.
\end{lemma}

\begin{proof}
    Note that since $k \geq \frac{8t}{\epsilon}$, we have 
    $\lceil \epsilon k/8 \rceil \geq t \geq t/2+1$.
    First we bound the probability that $\rk(U) \leq t/2$. 
    If $U$ has rank at most $t/2$, then there is a choice of $\lfloor t/2\rfloor$ columns of $U$ such that all the other ones are in the linear span of these selected columns.
    Since the linear span of $\lfloor t/2\rfloor$  vectors has dimension at most $t/2$, and so size at most $2^{t/2}$, we deduce the following.
    \[
    \everymath={\displaystyle}
    \renewcommand{\arraystretch}{2.5}
    \begin{array}{r l l}
        \Pr\left[\rk(U) \leq t/2\right] &\leq
        \binom{\lceil \epsilon k/8 \rceil}{\lfloor t/2\rfloor}\left(\frac{2^{\lfloor t/2\rfloor}-1}{2^t -1} \right)^{\lceil \epsilon k/8 \rceil -\lfloor t/2\rfloor } &\\
        &\leq \binom{\lceil\epsilon k/8\rceil}{\lfloor t/2\rfloor}\left(\frac{2^{t/2}-1}{2^t -1} \right)^{\epsilon k/8-t/2} & \\
        &\leq 2^{\epsilon k/8+1}\left(\frac{2^{t/2}-1}{2^t -1} \right)^{\epsilon k/16} & \text{ because $\frac{t}{2} \leq \frac{\epsilon k}{16}$ since } k\geq \frac{8t}{\epsilon} \\
        &\leq 2^{\epsilon k/4} (2^{t/2})^{-\epsilon k/16} & \\
        &\leq 2^{\epsilon t k/64} 2^{-\epsilon t k/32} & \text{ because } t \geq 16 \\
        &= 2^{- \epsilon t k/64}. & \\
    \end{array}    
    \] 
    Now we assume that $\rk(U)> t/2$.
    Then for every column $v$ of $V$, $v$ must be chosen in an affine space of dimension at most $t - \lfloor t/2\rfloor -1 \leq t/2$.
    It follows that
    \[
    \begin{split}
        \Pr[U^\top \cdot V = W \mid \rk(U) > t/2] & \leq 
        \left(\frac{2^{t/2}-1}{2^{t}-1} \right)^{\lceil\epsilon k/8 \rceil} \\
        &\leq (2^{-t/2})^{\epsilon k/8} \\
        & = 2^{-\epsilon t k/16}. \\
    \end{split}
    \]
    Therefore
    \[
    \begin{split}
        \Pr[U^\top \cdot V = W] & \leq \Pr\left[\rk(U) \leq t/2\right] + \Pr[U^\top \cdot V = W \mid \rk(U) > t/2]\\
        &\leq 2^{- \epsilon t k/64} + 2^{-\epsilon t k/16}\leq 2\cdot 2^{-\epsilon t k/64}.
    \end{split}
    \]
    We know that $\epsilon t k \geq 8 t^2 \geq 2 \cdot 64$, thus $2^{-\epsilon t k / 64} \leq 1/4$. As $2x \leq \sqrt{x}$ for any $x \in [0, 1/4]$, we end with $\Pr[U^\top \cdot V = W]  \leq 2 \cdot 2^{-\epsilon t k/64} \leq 2^{-\epsilon t k/128}$.
\end{proof}
For technical reasons, we prove the following seemingly weaker restatement of Theorem~\ref{thm:2+eps}.

\begin{theorem}\label{thm:2+eps+2}
    There exists a function $f\colon \mathbb{R}_{>0} \to \mathbb{N}$ such that for every $\epsilon>0$ and every positive integer $k$, 
    if $T$ is an $n$-vertex tournament with $n \geq 2k+ 2 \epsilon k+2$, then $\sinv_k(T) \leq f(\epsilon)$.
\end{theorem}

It is not difficult to see that Theorem~\ref{thm:2+eps+2} actually implies Theorem~\ref{thm:2+eps}. 
Indeed, given a function $f$ like in Theorem~\ref{thm:2+eps+2} at hand, define $f':\mathbb{R}_{>0} \to \mathbb{N}$ by $f'(\epsilon)=\max\{\frac{4}{\epsilon},f(\frac{\epsilon}{2})\}$. 
Let $T$ be a tournament with $|V(T)|\geq 2k+1 +\epsilon k$ for some positive integer $k$. 
If $k \leq \frac{2}{\epsilon}$, then Theorem~\ref{thm:M<2k} yields $\sinv_k(T)\leq 2k\leq \frac{4}{\epsilon}\leq f'(\epsilon)$. 
Otherwise, we have $|V(T)|\geq 2k+1+\epsilon k\geq 2k+2 + \frac{\epsilon}{2}k$, so $\sinv_k(T)\leq f(\frac{\epsilon}{2})\leq f'(\epsilon)$ by Theorem~\ref{thm:2+eps+2}.

\begin{proof}

Without loss of generality, we may assume $\epsilon \leq \frac{1}{3}$.
Let $C$ be a constant such that $\sinv_k(T) \leq 1$ if $n \geq Ck$, which exists by Theorem~\ref{nk1}. 
Let $t$ be the smallest integer such that 
\begin{itemize}
    \item $t\geq 16$, 
    \item $t \geq \log(1+\frac{48}{\epsilon})$, and
    \item $t \geq \frac{128}{\epsilon}(2C+2+\epsilon/4) +16$.
\end{itemize}
Clearly, $t$ is well defined and depends only on $\epsilon$.
Let $k_0$ be the smallest integer such that for every $k' \geq k_0$,
\begin{equation}\label{eq:condition_k_large_enough}
    (2^t-1) \exp\left(-\epsilon^2\frac{(2+\epsilon)k'}{24}\right) + 3(Ck')^2\exp\left(-\epsilon^2\frac{(2+\epsilon)k'}{4096}\right) + 2^{-k'} <1.
\end{equation}

We now prove the statement for $f(\epsilon)=\max\{t,2k_0-2,\lceil\frac{16t}{\epsilon}\rceil-2\}$. If $k < k_0$, then we conclude directly using Theorem~\ref{thm:M<2k} that $\sinv_k(T) \leq 2k\leq 2k_0-2\leq f(\epsilon)$. Similarly, if $k \leq \frac{8t}{\epsilon}-1$, then we conclude by Theorem~\ref{thm:M<2k} that $\sinv_k(T) \leq \frac{16t}{\epsilon}-2\leq f(\epsilon)$. 
 Moreover, if $n \geq Ck$, we have $\sinv_k(T)\leq 1 \leq f(\epsilon)$.
Henceforth, we may assume $k\geq \max\{k_0, \frac{8t}{\epsilon}-1\}$ and $n \leq Ck$.
\medskip


For every vertex $u \in V(T)$, we choose uniformly and independently at random a vector $\vec{u} \in \mathbb{F}_2^t \setminus \{\vec{0}\}$.
For $i\in [t]$, let $X_i = \{u \in V(T) \mid \vec{u}_i =1\}$.
We will prove that with positive probability, the tournament $T' = \Inv(T; X_1, \dots, X_t)$ is $k$-strong.
Note that for every arc $uv \in A(T)$, we have $uv\in A(T')$ if and only if $\vec{u} \cdot \vec{v}=0$.
For two disjoint subsets $A$ and $B$ of vertices of $T'$, a {\bf directed $(A,B)$-matching} is a set of arcs with tails in $A$, heads in $B$, and without common tail or common head.
\begin{claim}\label{claim:decompose_into_events_A_B_C_}
    If $T'$ is not $k$-strong, then at least one of the following events occurs:
    \begin{enumerate}[label=\Alph*]
        \item[$E_1$]: there is a vector $\vec{z} \in \mathbb{F}_2^t \setminus \{\vec{0}\}$ such that $|\{v \in V(T)\mid \vec{v} \neq \vec{z}\}|< k$,
        \item[$E_2$]: there are $u,v \in V(T)$ with $\vec{u} \neq\vec{v} $ such that $\min\{|N_{T'}^+(u)\cap N_{T'}^-(v)|,|N_{T'}^+(u)\cap N_{T'}^+(v)|,|N_{T'}^-(u)\cap N_{T'}^-(v)|\}\leq \left(1+\frac{\epsilon}{4}\right)\frac{k}{2}$,
        \item[$E_3$]: there are disjoint sets $A,B \subseteq V(T')$ with $|A|,|B| \geq \left(1+\frac{\epsilon}{4}\right)\frac{k}{2}$ 
        with no directed $(A,B)$-matching of size at least $\frac{k}{2}$.
    \end{enumerate}
\end{claim}

\begin{proofclaim}
    Assume that none of~$E_1$,~$E_2$ and~$E_3$ holds.
    Suppose for a contradiction that there is a set $X$ of at most $k-1$ vertices, and a partition $(V_1,V_2)$ of $V(T'-X)$ into nonempty sets such that $V_2 \Rightarrow V_1$ in $T'-X$.
    Since~$E_1$ does not hold, there exist $x,y\in V_1 \cup V_2$  with $\vec{x} \neq \vec{y}$. If both $x$ and $y$ are in $V_1$ (resp. $V_2$), consider $v \in V_2$ (resp. $u \in V_1$) and either $\vec{x}\neq\vec{v}$ or $\vec{y}\neq\vec{v}$ (resp. $\vec{x}\neq\vec{u}$ or $\vec{y}\neq\vec{u}$).
    If $x \in V_1$ and $y \in V_2$ we set $u=x$ and $v=y$,
    and if $y \in V_2$ and $x \in V_1$ we set $u=y$ and $v=x$.
    In all cases, there are $u\in V_1$ and $v \in V_2$ with $\vec{u} \neq \vec{v}$.

    Now, as ~$E_2$ does not hold, we have $|N_{T'}^+(u)\cap N_{T'}^-(v)|,|N_{T'}^+(u)\cap N_{T'}^+(v)|,|N_{T'}^-(u)\cap N_{T'}^-(v)|\geq(1+\epsilon/4)\frac{k}{2}$.
    Finally, as~$E_3$ does not hold, 
    there is a directed $(N_{T'}^+(u)\cap N_{T'}^+(v), N_{T'}^-(u)\cap N_{T'}^-(v))$-matching $M$ of size at least $k/2$ in $T'$.
    For every arc $e=xy\in M$, observe that $P_e=uxyv$ is a directed $(u,v)$-path in $T'$. 
    Furthermore, for every $x \in N_{T'}^+(u)\cap N_{T'}^-(v)$, observe that $P_x=uxv$ is a directed $(u,v)$-path in $T'$. 
    This yields a collection of at least $k$ internally vertex-disjoint $(u,v)$-paths in $T'$,
    a contradiction since every $(u,v)$-path meets $X$ which has size at most $k-1$.
\end{proofclaim}

We will show that with positive probability none of the events ~$E_1$,~$E_2$ and~$E_3$ occurs.

\begin{claim}\label{claim:proba_A}
    $\Pr(E_1) \leq (2^t-1) \exp\left(-\epsilon^2\frac{n}{16}\right)$
\end{claim}

\begin{proofclaim}
    If $\vec{c} \in \mathbb{F}_2^t \setminus\{\vec{0}\}$ is fixed, then $Y_{\vec{c}} = |\{u \in V(T) \mid \vec{u} \neq \vec{c}\}|$ is a random variable 
    having a binomial law with parameters $n$ and $1-\frac{1}{2^t-1}$. 
    As
    $\epsilon  \leq \frac{1}{3}$ and $t\geq 2$, we have $k \leq \frac{1}{2}n\leq \frac{2}{3}\cdot\frac{5}{6}n\leq(1-\frac{1}{2^t-1})(1-\frac{\epsilon}{2})n$. By Chernoff's Bound (Proposition~\ref{chernoff}), and because $t \geq 2$, we have
    \begin{align*}
    \Pr[\exists \vec{c} \in \mathbb{F}_2^t\setminus\{\vec{0}\}, Y_{\vec{c}} < k] 
    &\leq \sum_{\vec{c} \in \mathbb{F}_2^t\setminus\{0\}} \Pr[Y_{\vec{c}} < k] \\
    &\leq \sum_{\vec{c} \in \mathbb{F}_2^t\setminus\{0\}} \Pr\left[Y_{\vec{c}} < \left(1-\frac{\epsilon}{2}\right)\left(1-\frac{1}{2^t-1}\right)n\right] \\
    &\leq (2^t-1)\exp\left(-\left(\frac{\epsilon}{2}\right)^2\left(1-\frac{1}{2^t-1}\right)\frac{n}{2}\right) \\
    &\leq (2^t-1)\exp\left(-\epsilon^2 \frac{n}{16}\right),
    \end{align*}
    as claimed.
\end{proofclaim}



\begin{claim}\label{claim:proba_B}
    $\Pr(E_2) \leq 3n(n-1)\exp\left(-\epsilon^2\frac{n-2}{4096}\right)$
\end{claim}

\begin{proofclaim}
    Let $u,v$ be distinct vertices and let
    $A = N_{T'}^+(u)\cap N_{T'}^-(v), B = N_{T'}^+(u)\cap N_{T'}^+(v)$ and $C=N_{T'}^-(u)\cap N_{T'}^-(v)$.
    Let $X \in \{A,B,C\}$.
    Once $\vec{u}$ and $\vec{v}$ have been revealed, for every vertex $w \neq u,v$, let $Y_w$ be a random variable with $Y_w=1$ if $w \in X$, $Y_w=0$ otherwise.
    By Lemma~\ref{lemma:proba_2}, $Y_w$ is a random variable following a Bernoulli distribution whose parameter is at least $\frac{1}{4}(1-\frac{3}{2^t-1})$. Further, the $Y_w$ are mutually independent.

    We now define a random variable $X_w$ for every $w \in V(T)\setminus \{u,v\}$ as follows.
    If $Y_w=0$, set $X_w=0$, and otherwise set $X_w=1$ with probability $\frac{\frac{1}{4}(1-\frac{3}{2^t-1})}{\Pr[Y_w=1]}$ and $X_w=0$ otherwise, where the latter random experiments are executed independently. Observe that, as the $Y_w$ are mutually independent, so are the $X_w$. Moreover, $\Pr[X_w=1] = \frac{1}{4}(1-\frac{3}{2^t-1})$
    and $\sum_{w \in V(T) \setminus \{u,v\}} X_w \geq |X|$.
    
    We have $(1+\frac{\epsilon}{4})\frac{k}{2} \leq \frac{1+\frac{\epsilon}{4}}{2+\epsilon}(n-2)/2 = \frac{1}{2}\left(1-\frac{\frac{\epsilon}{4}}{1+\frac{\epsilon}{2}}\right) \frac{n-2}{2} \leq \frac{1}{4}\left(1-\frac{\epsilon}{8}\right)(n-2)$ since $n \geq (2+\epsilon)k+2$.
    Moreover, as $t \geq  \log(\frac{48}{\epsilon}+1)$ we have 
    $\frac{1}{4}(1-\epsilon/8)\leq \frac{1}{4}(1-\epsilon/16)^2 \leq \frac{1}{4}\left(1-\frac{3}{2^t-1}\right)(1-\frac{\epsilon}{16}) = \left(\frac{1}{4}-\frac{3}{4(2^t-1)}\right)\left(1-\frac{\epsilon}{16}\right)$.
    Hence $\left(1+\frac{\epsilon}{4}\right) \frac{k}{2} \leq \left(1-\frac{\epsilon}{16}\right)\left(\frac{1}{4}-\frac{3}{4(2^t-1)}\right)(n-2)$, and by Chernoff's Bound (Proposition~\ref{chernoff})
     \begin{align*}
        \Pr\left[|X| \leq \left(1+\frac{\epsilon}{4}\right)\frac{k}{2}\right]
        & \leq \Pr\left[|X| \leq \left(1-\frac{\epsilon}{16}\right)\left(\frac{1}{4}-\frac{3}{4(2^t-1)}\right)(n-2)\right]\\
        &\leq \exp\left(-\left(\frac{\epsilon}{16}\right)^2\left(\frac{1}{4}-\frac{3}{4(2^t-1)}\right)\frac{n-2}{2}\right)\\    
        &\leq \exp\left(-\epsilon^2\frac{n-2}{4096}\right) 
   \end{align*}
    since $t \geq 5$ implies $\frac{3}{2^t-1} \leq \frac{1}{8}$. 
    Hence, by the Union Bound (Proposition~\ref{union}), $\Pr\left[\min\{|A|,|B|,|C|\} \leq \left(1+\frac{\epsilon}{4}\right)\frac{k}{2}\right] \leq \sum_{X \in \{A,B,C\}}\Pr\left[|X| \leq \left(1+\frac{\epsilon}{4}\right)\frac{k}{2}\right]\leq 3\exp\left(-\epsilon^2\frac{n-2}{4096}\right)$.
    Finally, by the Union Bound over $u,v$, the probability that there exist $u,v$ distinct such that $\min\{|A|,|B|,|C|\} \leq \left(1+\frac{\epsilon}{4}\right)\frac{k}{2}$ is at most
     $3n(n-1)\exp\left(-\epsilon^2\frac{n-2}{4096}\right)$.
\end{proofclaim}

For two disjoint sets of vertices $X,Y$ in $T'$, we denote by $\mu_{T'}(X,Y)$ the size of a largest directed $(X,Y)$-matching in $T'$.

\begin{claim}\label{claim:proba_C}
    $\Pr(E_3) \leq 2^{-k}$.
\end{claim}
    
\begin{proofclaim}
    Let $A, B \subseteq V(T')$ be disjoint sets of $\lceil (1+\frac{\epsilon}{4})\frac{k}{2} \rceil$ vertices.
    We shall prove that with high probability there is a directed $(A,B)$-matching in $T'$ of size at least $\frac{k}{2}$.
    
    Let $M$ be a maximal directed $(A,B)$-matching and let $Y_A$ (resp. $Y_B$) the the set of vertices in $A$ (resp. in $B$) incident to no arc of $M$.
    Then $Y_B \Rightarrow Y_A$ in $T'$ since $M$ is maximal.
    Moreover, if $|M| \leq k/2$, then $|Y_A|,|Y_B| \geq (1+\frac{\epsilon}{4})\frac{k}{2}-\frac{k}{2} = \frac{\epsilon}{8}k$.
    
    For every such $Y_A \subseteq A, Y_B \subseteq B$, we identify $Y_A$ and $Y_B$ with the matrices whose columns are the $\vec{u}$ for $u \in Y_A$ (resp. $u \in Y_B$).    

    We also denote by $T(Y_B,Y_A)$ the $|Y_B| \times |Y_A|$ matrix over $\mathbb{F}_2$ whose cell $(b,a)$ equals $1$ if and only if $ab \in A(T)$.
    Then observe that $Y_B \Rightarrow Y_A$ in $T'$ if and only if $Y_B^\top \cdot Y_A = T(Y_B,Y_A)$.
    By these observations, we have
    \[
    \renewcommand{\arraystretch}{1.5}
    \begin{array}{r c l}
       \Pr[\mu_{T'}(A,B)<k]  & \leq & 
       \Pr[\exists Y_A \subseteq A, |Y_A| \geq \frac{\epsilon}{8}k, \exists Y_B \subseteq B, |Y_B| \geq \frac{\epsilon}{8}k, Y_B \Rightarrow Y_A \text{ in } T']  \\
       & \leq & \Pr[\exists Y_A \subseteq A, |Y_A| \geq \frac{\epsilon}{8}k, \exists Y_B \subseteq B, |Y_B| \geq \frac{\epsilon}{8}k, Y_B^\top \cdot Y_A = T(Y_B,Y_A)]  \\
       & \leq & 2^{2\left\lceil (1+\frac{\epsilon}{4})\frac{k}{2} \right\rceil} 2^{-\frac{\epsilon t k}{128}} \hspace{1cm} \text{ by Lemma~\ref{lemma:proba_3} and the Union Bound} \\
       & \leq & 2^{\left(1+\frac{\epsilon}{4}\right)k+2} 2^{-\frac{\epsilon t k}{128}} \hspace{1.325cm}  \\
       & \leq & 2^{-(2C+1)k} \\
    \end{array}
    \]
    
    as $t \geq \frac{128}{\epsilon}\left(2C+2+\frac{\epsilon}{4}\right) + 16$ and $\frac{\epsilon k}{64} \geq \frac{1}{8}$.
    It follows from the Union Bound (Proposition~\ref{union}) that $\Pr(E_3) \leq 2^{2n} 2^{-(2C+1)k} \leq 2^{-k}$ using the fact that $n \leq Ck$.
\end{proofclaim}

We can now conclude using Claims~\ref{claim:decompose_into_events_A_B_C_},~\ref{claim:proba_A},~\ref{claim:proba_B}
and~\ref{claim:proba_C} and the Union Bound: 
\[
\begin{split}
    \Pr[T' \text{ not } k\text{-strong}] &\leq \Pr(E_1) + \Pr(E_2)+\Pr(E_3) \\
    &\leq (2^t-1) \exp\left(-\epsilon^2\frac{n}{16}\right) + 3n^2\exp\left(-\epsilon^2\frac{n-2}{4096}\right) + 2^{-k} \\
    &\leq (2^t-1) \exp\left(-\epsilon^2\frac{(2+\epsilon)k}{16}\right) + 3(Ck)^2\exp\left(-\epsilon^2\frac{(2+\epsilon)k}{4096}\right) + 2^{-k} \\
    &<1, \\
\end{split}
\]
by \eqref{eq:condition_k_large_enough}.
This proves that there exist $X_1, \dots, X_t \subseteq V(T)$ such that $T'=\Inv(T;X_1,\dots, X_t)$ is $k$-strong.
\end{proof}

\section{Conclusion}\label{conclusion}

In this paper, we investigated the minimum number of inversions needed to make a digraph $k$-strong or $k$-arc-strong.
Many open questions remain.

\medskip

From an algorithmic perspective, it would be interesting to gain a deeper understanding of the complexity when the problem is restricted to tournaments. While the complexity for fixed $k$ is settled by Corollary~\ref{cor:sinv_k-poly}, the following question remains open:

\begin{problem}
    What is the complexity of computing $\sinv_k(T)$ (resp. $\sinv'_k(T)$) for a given tournament $T$ if $k$ is part of the input ?
\end{problem}

\medskip
Regarding bounds, we provided many on $\sinv_k(n)$ and $\sinv'_k(n)$. But those functions still need to be better understood. 

Firstly, it would be nice to better understand the asymptotic behaviour of $\sinv_k(n)$ and determine whether an analogue of Theorem~\ref{thm:extrem} exists for $\sinv_k$.
\begin{problem}
  Find good lower and upper bounds on $\sinv_k(n)$.
\end{problem}

Secondly, we posed Conjecture~\ref{conjM} which we restate. 
\theconj*
Conjecture~\ref{conjM} is motivated by the following consideration. By Proposition~\ref{4k-2}, every tournament $T$ of order $n>4k-2$ has a vertex $v$ with $d_T^-(v)\geq k$ and $d_T^+(v)\geq k$. 
Moreover, adding a vertex with in- and out-degree at least $k$ to a $k$-(arc-)strong digraph results in a $k$-(arc-)strong digraph. 
Hence any $k$(-arc)-strengthening family of $T-v$ is also a $k$-(arc-)strengthening family of $T$. Thus $\sinv_k(T) \leq \sinv_k(T-v) \leq m_k(n-1) $ and $\sinv'_k(T) \leq \sinv'_k(T-v) \leq m'_k(n-1) $. Hence $m_k(n) \leq m_k(n-1)$ and $m'_k(n) \leq m'_k(n-1)$.
Therefore, in order to approach Conjecture~\ref{conjM} and Problem~\ref{rsedth}, 
it is sufficient to consider tournaments whose order is in the range from $2k+1$ to $4k-2$.

\medskip

Another intriguing question is whether $M_k = M'_k$.
\begin{problem}
    Is it true that $M_k = M'_k$ for every positive integer $k$ ?
\end{problem}
In Subsection~\ref{sec:52}, we proved that this is the case for $k=1,2$.

\medskip

Furthermore, there is still a significant gap between the logarithmic lower bounds and the linear upper bounds on $M'_k$ and $M_k$. Therefore it would be good to improve these.
A first question is the following.
\begin{problem}
    Are $M'_k$ and $M_k$ sublinear functions of $k$ ? 
\end{problem}

We also proved a collection of bounds for $N_k(1)$ and $N_k(3)$. It would be interesting to have stronger bounds on $N_k(i)$ for any integer $i$. In particular, a tight bound for $N_k(1)$ would be satisfying.

Finally, Theorem~\ref{thm:2+eps} states that every tournament that is a constant factor bigger than $2k$ can be made $k$-strong by a constant number of inversions. It would be interesting to know if this result can be strengthened in the following way:

\begin{problem}
    Is there an integer $t$ such that for every $\epsilon>0$, there is an integer $k_0$ such that for every $k \geq k_0$, for every tournament $T$ on at least $(2+\epsilon)k$ vertices, we have $\sinv_k(T)\leq t$?
\end{problem}

Proposition~\ref{nk1unter} shows that such an integer would have to satisfy $t \geq 2$, so $t=2$ is the first open case. The analogous statement for $\sinv'_k$ is also open.

\bibliographystyle{alpha}
\bibliography{biblio}

\end{document}